\documentclass[11pt]{amsart}

\usepackage{amssymb}
\usepackage{amsmath}
\usepackage{amsthm}
\usepackage{color} 
\usepackage{verbatim}
\usepackage{bbm}
\usepackage{dsfont}
\usepackage{mathtools}

\newtheorem{theorem}{Theorem}[section]
\newtheorem{lemma}{Lemma}[section]
\newtheorem{remark}{Remark}[section]
\newtheorem{assumption}{Assumption}[section]
\newtheorem{corollary}{Corollary}[section]

\newcommand{\sfa}{\sup_{\alpha.\in\mathfrak A}}
\newcommand{\sfb}{\sup_{\beta.\in\mathfrak B}}
\newcommand{\ifb}{\inf_{\beta.\in\mathfrak B}}
\newcommand{\sfm}{\sup_{\boldsymbol\mu\in\mathcal M}}
\newcommand{\ifn}{\inf_{\boldsymbol\nu\in\mathcal N}}

\newcommand{\Rd}{\mathbb{R}^d}
\newcommand{\Sd}{\mathbb{S}^d}

\newcommand{\bSdp}{\overline{{\mathbb{S}}^d_+}}

\newcommand{\tr}{\operatorname{tr}}

\newcommand{\B}{\mathrm B}

\newcommand{\pxsop}{\frac{\psi_{(\xi)}^2}{\psi}}

\newcommand{\pxtsops}{\frac{\psi_{(\xi_t)}^2}{\psi^2}}

\newcommand{\vxosbw}{\frac{|v_{(\xi)}(x)|}{\sqrt{\mathrm{B}_1(x,\xi)}}}

\numberwithin{equation}{section}

\begin{document}
\title[Isaacs equations and nonconvex Hessian equations]{On representation and regularity of viscosity solutions to degenerate Isaacs equations and certain nonconvex Hessian equations}
\author{Wei Zhou}
\address{School of Mathematics, University of Minnesota}
\email{zhoux123@math.umn.edu}

\begin{abstract}
We study the smoothness of the upper and lower value functions of stochastic differential games in the framework of time-homogeneous (possibly degenerate) diffusion processes in a domain, under the assumption that the diffusion, drift and discount coefficients are all independent of the spatial variables. Under suitable conditions (see Assumptions \ref{geod} and \ref{invc}), we obtain the optimal local Lipschitz continuity of the value functions, provided that the running and terminal payoffs are globally Lipschitz. As applications, we obtain the stochastic representation and optimal interior $C^{0,1}$-regularity of the unique viscosity solution to the Dirichlet problem for certain degenerate elliptic, nonconvex Hessian equations in suitable domains, with Lipschitz boundary data.
\end{abstract}

\maketitle

\section{Introduction}\label{section1}

\subsection{Background} The point of this paper is to study the smoothness of the upper and lower value functions of stochastic differential games and its applications to the regularity theory of fully nonlinear, nonconvex elliptic equations. 

To be more specific, on the game theory aspect, we consider the zero-sum, two-player, stochastic differential games in the framework of time-homogeneous, possibly degenerate, controlled diffusion processes driven up to the first exit time from a smooth bounded domain, and study the smoothness of the corresponding upper and lower value functions. On the PDE theory aspect, we are concerned with the regularity of viscosity solutions to the Dirichlet problem for Isaacs equations, which are fully nonlinear, second order, possibly degenerate, nonconvex elliptic equations. 

For the general background of mathematically rigorous theory of differential games, we refer to the historical remarks in Section XI.10 in \cite{MR2179357}. As far as the smoothness of value functions of stochastic differential games are concerned, in our settings, the diffusion processes are considered up to the first exit time of a bounded domain, so our game is of random time horizon, with non-vanishing terminal payoff.  We refer to \cite{MR2500878,Krylov:2012uq} for previous work on the smoothness of the value function in this situation. The relation between weak solutions to fully nonlinear, nonconvex elliptic equations and value functions in stochastic differential games are studied on various aspects in \cite{MR997385,MR2500878,MR3062445,Krylov:2012fk,MR934730,MR945913,MR1422779}.

For the fully nonlinear elliptic equations
$$F(D^2u, Du, u, x)=f $$
without convexity assumptions on $F$, the regularity theory is not so well-investigated, compared with convex ones. With the assumption of non-degeneracy, under various conditions on $F$, the $C^{1,\alpha}$-regularity has been obtained in \cite{MR1005611, MR931007, MR1034037,MR1606359}, and recently in \cite{conealpha} (cf. the Introduction Section in \cite{conealpha} for the comparison on the $C^{1,\alpha}$-regularity results in the above mentioned papers.) 
Without the assumption of  non-degeneracy on $F$, when $F$ is positively homogeneous with degree one, several Lipschitz continuity results has been established by probabilistic approach. In \cite{MR2500878}, Kovats showed that the value function, which is a viscosity solution of the associated Isaacs equations, is (globally) Lipschitz under the assumptions that the terminal payoff  (i.e. the boundary data) is in the class of $C^2$ and the lower bound of the discount factors (i.e. zero-order coefficients) is  large compared to the Lipschitz constant for the diffusion and drift coefficients (i.e. second and first order coefficients). In \cite{Krylov:2012uq}, in a variety of settings, Krylov obtained Lipschitz continuity and estimates of the Lipschitz constant of value functions, one of which is independent of the constant of non-degeneracy. 

It is also worth mentioning that for uniformly nondegenerate, nonconvex elliptic equations, the optimal regularity of  viscosity solutions is $C^{1,\alpha}$, where $\alpha$ can be very small.  See \cite{counterex} for the example establishing the optimality due to Nadirashvili and Vl{\u{a}}du{\c{t}} (for dimension $d\ge 5$). Therefore in the author's opinion, a plausible conjecture is that for degenerate nonconvex elliptic equations, in general, the best regularity of  viscosity solutions is the $C^{0,1}$-regularity, even if the boundary data is smooth.

Another motivation of this work comes from our interest in the representation and regularity theory for degenerate Hessian equations in intermediate elliptic branches (i.e. the elliptic branches neither convex nor concave). Recall that the so-called Hessian equations   (in the simplest case) with Dirichlet boundary condition are defined as
\begin{equation}\label{Hess}
\left\{
\begin{array}{rcll}
F(\operatorname{Hess}u)&=&0 &\mbox{in $D$}\\
u&=&\varphi&\mbox{on $\partial D$},
\end{array}
\right.
\end{equation}
where $F$ depends only on the eigenvalues of the Hessian of $u$. The partial differential equation in (\ref{Hess}) is a fully-nonlinear, pure second-order, constant-coefficient equation which is stable under the action of the orthogonal group by conjugation, and under appropriate conditions, it has multiple elliptic branches due to the nonlinearity,  see \cite{MR1284912}. The two outermost branches are convex or concave, but all intermediate ones are neither convex nor concave. Unfortunately, none of the aforementioned regularity results is applicable here, if we focus on degenerate elliptic, intermediate branches. In \cite{MR2487853}, Harvey and Lawson obtained that under a general condition on the geometry of the domain, if the boundary data $\varphi\in C(\partial D)$, then the Dirichlet problem (\ref{Hess}) in each elliptic branch has a unique viscosity solution in the class of $C(\bar D)$. (Actually their theorem covers much more general equations.) It seems to the author that in degenerate intermediate branches, on the regularity of the viscosity solution, nothing more is known besides the continuity.

Now we are in a position to discuss our setups and results. Our goal is to establish the local Lipschitz continuity of viscosity solutions to the Dirichlet problem for degenerate Isaacs equations with Lipschitz boundary data. Our approach is probabilistic by investigating the smoothness of stochastic representations of viscosity solutions, i.e. value functions of the corresponding stochastic differential games. Since currently we are not able to tackle the most general cases, we wish that, at least,  our results should be applicable to the nonconvex degenerate Hessian equations. 

To this end, for the diffusion processes in our stochastic differential games, we do not assume non-degeneracy on diffusion coefficients or strict positivity on discount coefficients. Instead, we require that the diffusion, drift and discount coefficients are all independent of the spatial variables. Under a natural condition on the geometry of the domain (i.e. Assumption \ref{geod}) and an invariance property of the coefficients (i.e. Assumption \ref{invc}), we first obtain the local Lipschitz continuity of the value functions, when the running and terminal payoffs are both (globally) Lipschitz. These value functions are the viscosity solution to their associated dynamic programming equations, i.e. Isaacs equations, with Dirichlet boundary condition. Note that under our global $C^{0,1}$-regularity assumption on the terminal payoff (i.e. boundary data), 
the interior $C^{0,1}$-regularity of the value function (i.e. viscosity solution to its associated degenerate Isaacs equation) is optimal. As examples of application, we then obtain the stochastic representation and optimal interior $C^{0,1}$-regularity of the unique viscosity solution to the Dirichlet problem for certain nonconvex degenerate Hessian equations in strictly convex domains, with Lipschitz boundary data.

\subsection{Outline of the paper} This paper is organized as follows. In Section \ref{section2} we introduce our setup and state our main theorems. Four auxiliary lemmas are given in Section \ref{section3}. In Section \ref{section4} we prove our main theorems, and in Section \ref{section5} we discuss applications to nonconvex degenerate Hessian equations, which is followed by  Section \ref{section6} about further directions.

\subsection{Basic notation} Throughout the paper, the summation convention for repeated indices is assumed, even when both repeated indices appear in the superscript. We usually put the indices in the superscript, since the subscript is for the temporal variable of stochastic processes. We use the following notation.

Denote by $\Rd$ the $d$-dimensional Euclidean space with points $x=(x^1, \cdots, x^d)$, $x^i\in\mathbb R$ (real numbers); $x\cdot y=(x,y)=x^iy^i$ is the inner product for $x,y\in \Rd$, and $|x|^2:=(x,x)$. Denote by $\mathbb M^{m\times n}=\mathbb M^{m\times n}(\mathbb R)$ the space of $m\times n$-size matrices with real entries, with element $\sigma=(\sigma^{ij})_{m\times n}$. $\sigma^*$ represents the transpose of $\sigma$. 
$\|\sigma\|^2:=\tr(\sigma\sigma^*)$. Denote by $\mathbb S^d$, $\mathbb S^d_+$, $\overline{ \mathbb S_+^d}$, $\mathbb O^d$  the sets of symmetric, nonnegative symmetric, positive symmetric, orthogonal matrices, respectively.

For any $s,t\in\mathbb R$, let $s\wedge t=\max\{s,t\}$ and $s\vee t=\min\{s,t\}$.

For any sufficiently smooth function from $\Rd$ to $\mathbb R$, $u_x=(u_{x^1},\cdots,u_{x^d})$ is the gradient of $u$ and $u_{xx}=(u_{x^ix^j})_{d\times d}$ is the Hessian matrix of second derivatives $u_{x^ix^j}$, $i,j=1,\cdots d$. Set
$$u_{(y)}=u_{x^i}y^i,\quad u^2_{(y)}=(u_{(y)})^2.$$

For each nonnegative integer $k$ or $k=\infty$, $C^k(D)=C^k_{loc}(D)$  is the set of functions having all partial derivatives of order $\le k$ continuous in  $D$. $C^{k}(\bar D)$ is the set of functions in $C^k(D)$ all of whose partial derivatives of order $\le k$ have continuous extension to $\bar D$. 
\begin{align*}
\|u\|_{C^0(\bar D)}&=|u|_{0,D}:=\sup_{x\in D}|u(x)|,\\
\|u\|_{C^1(\bar D)}&=|u|_{1,D}:=|u|_{0,D}+|u_x|_{0,D}, \quad\mbox{etc..}
\end{align*}
The H\"older spaces $C^{k,\gamma}(D)=C_{loc}^{k,\gamma}(D)$ (resp. $C^{k,\gamma}(\bar D)$) are the subspaces of $C^{k}(D)$ (resp. $C^k(\bar D)$) consisting of functions whose $k$-th order partial derivatives are locally H\"older continuous (resp. globally H\"older continuous) with exponent $\gamma$ in $D$, where $0<\gamma\le 1$. In particular, when $k=0$, $\gamma=1$, we have $C^{0,1}(D)=C_{loc}^{0,1}(D)$ (resp. $C^{0,1}(D)$), the space of locally Lipschitz functions (resp. globally Lipschitz functions) in $D$, with
\begin{align*}
\|u\|_{C^{0,1}(\bar D)}=|u|_{0,1,D}:=|u|_{0,D}+\sup_{\substack{x,y\in D\\x\ne y}}\frac{|u(x)-u(y)|}{|x-y|}.
\end{align*}

\section{Main theorems}\label{section2}
\subsection{Settings} 

The setup of our two-player, zero-sum stochastic differential game is the following. 
Let $d$ and $d_1$ be integers, and $A$ and $B$ be separable metric spaces. Assume that the following continuous and bounded functions on $A\times B$ are given:
\begin{itemize}
\item $\mathbb M^{d\times d_1}$-valued function $\sigma^{\alpha\beta}=(\sigma^{\alpha\beta}_1,\cdots,\sigma^{\alpha\beta}_{d_1})$,
\item $\Rd$-valued function $b^{\alpha\beta}$,
\item $\mathbb R$-valued non-negative function $c^{\alpha\beta}$.
\end{itemize}
Assume also that $\sigma^{\alpha\beta},b^{\alpha\beta},c^{\alpha\beta}$ are continuous with respect $\alpha$ (resp. $\beta$) uniformly with respect  to $\beta$ (resp. $\alpha$).

Let $(\Omega,\mathcal{F},P)$ be a complete probability space, $\{\mathcal{F}_t;t \ge 0\}$ be an increasing filtration of $\sigma $-algebras $\mathcal{F}_t \subset \mathcal{F}$ which are complete with respect to $(\mathcal{F},P)$, and $(w_t ,\mathcal{F}_t ;t \ge 0)$ be a $d_1$-dimensional Wiener process on $(\Omega,\mathcal{F},P)$. 

Denote by $\mathfrak{A}$ (resp., $\mathfrak B$) the set of progressively measurable $A$-valued (resp. $B$-valued) processes $\alpha_t=\alpha_t(\omega)$ (resp. $\beta_t=\beta_t(\omega)$). In our game, $\mathfrak A$ and $\mathfrak B$ are the sets of policies. 

Then we proceed by introducing the sets of stratagies. Denote by $\mathcal M$ the set of $\mathfrak A$-valued functions $\boldsymbol\mu(\beta.)$ on $\mathfrak B$ such that, for any $T\in (0,\infty)$ and $\beta^1_\cdot, \beta^2_\cdot\in\mathfrak B$ satisfying
$$P\big(\beta_t^1=\beta_t^2 \mbox{ for almost all } t\le T\big)=1,$$
we have
$$P\big((\boldsymbol\mu(\beta^1_\cdot))_t=(\boldsymbol\mu(\beta^2_\cdot))_t\mbox{ for almost all } t\le T\big)=1.$$
Similarly, denote by $\mathcal N$ the set of $\mathfrak B$-valued functions $\boldsymbol\nu(\alpha.)$ on $\mathfrak A$ satisfying an analogous condition.

Let $D$ be a smooth bounded domain in $\Rd$ described by a smooth function $\psi$ which is non-singular on $\partial D$, i.e.
\begin{equation}\label{domain}
D:=\{x\in\Rd:\psi(x)>0\},\qquad|\psi_x|\ge1\mbox{ on }\partial D.
\end{equation}
We say that the function $\psi$ is global barrier of the domain $D$.

For the sake of definiteness, we suppose that
$$|\psi|_{3,D},\|\sigma\|_{0,A\times B},|b|_{0,A\times B}, |c|_{0,A\times B}\le K_0,$$
where $K_0\in[1,\infty)$ is constant.

On $\partial D$, a (globally) Lipschitz function $g$ is given. On the set $A\times B\times D$, a real-valued function $f=f^{\alpha\beta}(x)$ is defined, which is bounded and Borel measurable on $A\times B\times D$. Also, assume that for each $(\alpha,\beta)\in A\times B$, the function $f^{\alpha\beta}(\cdot)$ is (globally) Lipschitz in $\bar D$, and
$$\sup_{(\alpha,\beta)\in A\times B}\big|f^{\alpha\beta}\big|_{0,1,D},|g|_{0,1,\partial D}\le K_0.$$

Now we consider the time-homogeneous stochastic differential game in the framework of diffusion processes in which $D$ is the domain, $A$ and $B$ are the control sets of the two players, respectively, $\mathfrak A$ and $\mathfrak B$ are the sets of policies, $\mathcal M$ and $\mathcal N$ are the sets of strategies, $\sigma^{\alpha\beta}$, $b^{\alpha\beta}$, $c^{\alpha\beta}$ are the diffusion, drift and discount coefficients, and $f^{\alpha\beta}(x)$ and $g(x)$ are the running payoff and terminal payoff, respectively. 

More precisely, for each $(\alpha.,\beta.)\in\mathfrak{A}\times\mathfrak B$ and $x\in D$, the (possibly degenerate) diffusion process with the initial point $x$, under the policies $\alpha.$ and $\beta.$ is given by 
\begin{equation}\label{itox}
x_t^{\alpha.\beta.,x}=x+\int_0^t\sigma^{\alpha_s\beta_s}dw_s+\int_0^tb^{\alpha_s\beta_s}ds.
\end{equation}
If we set
\begin{equation}\label{phidiscount}
\phi_t^{\alpha.\beta.}=\int_0^tc^{\alpha_s\beta_s}ds,
\end{equation}
and let $\tau^{\alpha.\beta.,x}$ be the first exit time of $x_t^{\alpha.\beta.,x}$ from $D$, namely, 
$$\tau^{\alpha.\beta.,x}=\inf\{t\ge0:x_t^{\alpha.\beta.,x}\notin D\},$$
then the payoff of the stochastic differential game with the initial point $x\in D$, under the policies $\alpha.\in\mathfrak A$ and $\beta.\in\mathfrak B$ is defined as
\begin{equation*}
v^{\alpha.\beta.}(x)=E\bigg[\int_0^{\tau^{\alpha.\beta.,x}}f^{\alpha_t\beta_t}\big(x_t^{\alpha.\beta.,x}\big)e^{-\phi^{\alpha.\beta.}_t}dt+g\big(x^{\alpha.\beta.,x}_{\tau^{\alpha.\beta.,x}}\big)e^{-\phi^{\alpha.\beta.}_{\tau^{\alpha.\beta.,x}}}\bigg].
\end{equation*}
To make the expression shorter, we use abbreviated notation which is widely-adopted in stochastic control and game theory, according to which we put the superscripts $\alpha.$, $\beta.$ and $x$ beside the expectation sign instead of explicitly exhibiting them inside the expectation sign for every objects that can carry all or part of them. As a result, we write
\begin{equation}\label{payoff}
v^{\alpha.\beta.}(x)=E^{\alpha.\beta.}_x\bigg[\int_0^{\tau}f^{\alpha_t\beta_t}\big(x_t\big)e^{-\phi_t}dt+g\big(x_\tau\big)e^{-\phi_{\tau}}\bigg].
\end{equation}

There are two players I and II in the game. Player I chooses $\alpha.\in\mathfrak A$ and wishes to maximize the payoff $v^{\alpha.\beta.}(x)$ over all policies in $\mathfrak B$, while player II chooses $\beta.\in\mathfrak B$ and wishes minimize $v^{\alpha.\beta.}(x)$ over all policies in $\mathfrak A$. In the upper game, player I is allowed to know $\beta.$ before choosing $\alpha.$, while in the lower game, player II is allowed to know $\alpha.$ before choosing $\beta.$. This is how the strategies $\boldsymbol\mu(\beta.)\in\mathcal M$ and $\boldsymbol\nu(\alpha.)\in\mathcal N$ come into the picture. The upper value (function of $x$) of the game, corresponding to the upper game, is known as 
\begin{equation}\label{uppervalue}
v^+(x)=\sfm\ifb v^{\boldsymbol\mu(\beta.)\beta.}(x),
\end{equation}
while the lower value (function of $x$) of the game, corresponding to the lower game, is known as
\begin{equation}\label{lowervalue}
v^-(x)=\ifn\sfa v^{\alpha.\boldsymbol\nu(\alpha.)}(x).
\end{equation}

In the theory of fully nonlinear elliptic partial differential equations, it is well-known that the associated dynamic programming equation of $v^+$ and $v^-$ are the (possibly degenerate) upper and lower Isaacs equations with Dirichlet boundary conditions:

\begin{equation}\tag{{\bf{I}{$^+$}}}\label{ip}
\left\{
\begin{array}{rcll}
H^+(u_{xx},u_x,u,x)&=&0 &\text{in } D\qquad\\
u&=&g &\text{on }\partial D,
\end{array}
\right. 
\end{equation}

\begin{equation}\tag{{\bf{I}{$^-$}}}\label{im}
\left\{
\begin{array}{rcll}
H^-(u_{xx},u_x,u,x)&=&0 &\text{in } D\qquad\\
u&=&g &\text{on }\partial D,
\end{array}
\right. 
\end{equation}
respectively, where
\begin{align*}
H^+(u_{xx},u_x,u,x)=&\inf_{\beta\in B}\sup_{\alpha\in A}H^{\alpha\beta}(u_{xx},u_x,u,x),\\
H^-(u_{xx},u_x,u,x)=&\sup_{\alpha\in A}\inf_{\beta\in B}H^{\alpha\beta}(u_{xx},u_x,u,x),
\end{align*}
with
$$
H^{\alpha\beta}:\Gamma\rightarrow\mathbb R;\
(\gamma,p,z,x)\mapsto\tr(a^{\alpha\beta}\gamma)+b^{\alpha\beta}\cdot p-c^{\alpha\beta} z+f^{\alpha\beta}(x),$$
$$\Gamma=\Sd\times\Rd\times\mathbb R\times D,$$
$$
a^{\alpha\beta}=(1/2)\sigma^{\alpha\beta}(\sigma^{\alpha\beta})^*.
$$

The readers understand that the opposite of the upper value function is a lower value function of the game with opposite running and terminal payoffs, because
$$-v^+(x)=-\sfm\ifb v^{\boldsymbol\mu(\beta.)\beta.}(x)=\inf_{\boldsymbol\mu\in\mathcal M}\sfb \Big[-v^{\boldsymbol\mu(\beta.)\beta.}(x)\Big].$$
For the same reason, the Dirichlet problems (\ref{ip}) and (\ref{im}) are convertible from one to the other.

\subsection{Smoothness of value functions}

We are now in a position to state our main assumptions and theorems. We only consider the upper value function $v^+$ and its associated upper Isaacs equation $H^+=0$. The reader should be able to obtain analogous results on the lower value function $v^-$ and its associated lower Isaacs equation $H^-=0$.

\begin{assumption}\label{geod}(Geometry of the domain)
For each $x\in D$ and $(\alpha,\beta)\in A\times B$, we have 
\begin{equation}
L^{\alpha\beta}\psi:=(a^{\alpha\beta})_{ij}\psi_{x^ix^j}+ (b^{\alpha\beta})_i\psi_{x^i}\le-1.
\end{equation}
\end{assumption}

\begin{remark}
Assumption \ref{geod} ensures that the upper  value function given by (\ref{uppervalue}) is bounded and continuous.
\end{remark}

\begin{assumption}\label{invc}(Invariance property of the coefficients) If for each $q\in\mathbb O_d$, we define
\begin{equation}
H^{\alpha\beta q}(\gamma,p,z,x)=\tr(qa^{\alpha\beta}q^*\gamma)+b^{\alpha\beta}\cdot p-c^{\alpha\beta} z+f^{\alpha\beta}(x),
\end{equation}
then for all $(\gamma,p,z,x)\in\Gamma$, either
\begin{equation}\label{othinv}
\inf_{(\beta,q)\in B\times\mathbb O_d}\sup_{\alpha\in A}H^{\alpha\beta q}(\gamma,p,z,x)=\inf_{\beta\in B}\sup_{\alpha\in A}H^{\alpha\beta}(\gamma,p,z,x),
\end{equation}
or
\begin{equation}\label{othinv2}
\inf_{\beta\in B}\sup_{(\alpha,q)\in A\times\mathbb O_d}H^{\alpha\beta q}(\gamma,p,z,x)=\inf_{\beta\in B}\sup_{\alpha\in A}H^{\alpha\beta}(\gamma,p,z,x).
\end{equation}
holds.
\end{assumption}

\begin{remark}
Notice that 
$$\tr(qa^{\alpha\beta}q^*\gamma)=\tr(a^{\alpha\beta}q^*\gamma q),$$
so when 
\begin{equation}
H^+(q^*\gamma q,p,z,x)=H^+(\gamma,p,z,x), \quad\forall q\in \mathbb O_d,
\end{equation}
Condition (\ref{othinv}) holds. This means, in particular, Assumption \ref{invc} is satisfied in the case of Hessian equations, i.e., when $H^+$ depends on $\gamma$ only with respect to its eigenvalues.
\end{remark}

\begin{theorem}\label{thm1} If Assumptions \ref{geod} and \ref{invc} hold, then the upper value function $v^+$ given by (\ref{uppervalue}) is in the class of $C^{0,1}(D)\cap C(\bar D)$, and for a.e. $x\in D$, 
\begin{equation}
\big|(v^+)_{(\xi)}\big|\le N\big(|\xi|+|\psi_{(\xi)}|\psi^{-1/2}\big),\quad\forall\xi\in\Rd,
\end{equation}
where the constant $N=N(K_0,d,d_1)$. Moreover, $v^+$ is the viscosity solution to the Dirichlet problem for the upper Isaacs equation (\ref{ip}). 
\end{theorem}

\section{Auxiliary results}\label{section3}

We prove four auxiliary lemmas in this section.

\subsection{Uniform boundedness of expectation of first exit time of the diffusion processes} The following lemma will be used to prove that the upper value function given by (\ref{uppervalue}) is bounded and continuous.

\begin{lemma}\label{lemma0}
If Assumption \ref{geod} holds, then for each $x\in D$,
\begin{align}
\sfa \sfb E^{\alpha.\beta.}_x\tau^n\le n!|\psi|^{n-1}_{0,D}\psi(x), \ \forall n\in\mathbb N.
\end{align}
\end{lemma}

\begin{proof}
It suffices to prove the inequality for any $\alpha.\in\mathfrak A$, $\beta.\in\mathfrak B$ and notice that
\allowdisplaybreaks\begin{align*}
E^{\alpha.\beta.}_x\tau\le& -E^{\alpha.\beta.}_x\int_0^{\tau}L\psi dt=\psi(x)-E^{\alpha.\beta.}_x\psi(x_{\tau})=\psi(x),\\
E^{\alpha.\beta.}_x\tau^n=&nE_x^{\alpha.\beta.}\int_0^\infty(\tau-t)^{n-1}\mathbbm 1_{\tau>t}dt\\
=&nE_x^{\alpha.\beta.}\int_0^\infty \mathbbm 1_{\tau>t}E_x^{\alpha.\beta.}(\tau^{\alpha.\beta.,x_t})^{n-1}dt\\
\le& n\sup_{y\in D}E^{\alpha.\beta.}_y\tau^{n-1}\cdot E^{\alpha.\beta.}_x\tau\\
\le& n\sup_{y\in D}E^{\alpha.\beta.}_y\tau^{n-1}\cdot\psi(x).
\end{align*}
\end{proof}

\subsection{Approximating value functions and their regularity}
For all $0<\delta<1$, define 
\begin{equation}\label{diffusione}
x_t^{\alpha.\beta.,x}(\delta)=x+\int_0^t\sigma^{\alpha_s\beta_s}dw_s+\int_0^t\delta I d\tilde w_s+\int_0^tb^{\alpha_s\beta_s}ds,
\end{equation}
where $\tilde w_t$ is a $d$-dimensional Wiener process independent of $w_t$ and $I$ is the identity matrix of size $d\times d$. Let $\tau^{\alpha.\beta.,x}(\delta)$ be the first exit time of $x_t^{\alpha.\beta.,x}(\delta)$ from $D$. Also, let $g^{(\delta)}$ and $(f^{\alpha\beta})^{(\delta)}$ be real-valued $C^\infty(\bar D)$-functions satisfying
$$|g^{(\delta)}|_{1,D}\le 2|g|_{0,1,D}, \ |(f^{\alpha\beta})^{(\delta)}|_{1,D}\le 2|f^{\alpha\beta}|_{0,1,D},\quad\forall \delta\in(0,1),$$
$$\lim_{\delta\downarrow0}|g^{(\delta)}-g|_{0,D}=\lim_{\delta\downarrow0}\sup_{(\alpha,\beta)\in A\times B}|(f^{\alpha\beta})^{(\delta)}-f|_{0,D}=0.$$
Such regularization of $g$ and $f^{\alpha\beta}$ can be obtained by first extending them to some domain $D'\supset D$ and then convoluting them with some mollifier.

Consider the upper value function $v^+$ given by (\ref{payoff}) and (\ref{uppervalue}) in which $x_t^{\alpha.\beta.,x}$ and $g$ are replaced with $x_t^{\alpha.\beta.,x}(\delta)$ and $g^{(\delta)}$, respectively, i.e.
\begin{equation}
v^{(\delta)}(x)=\sfm \ifb (v^{\boldsymbol\mu(\beta.)\beta.})^{(\delta)}(x),
\end{equation}
with
\begin{equation}
(v^{\alpha.\beta.})^{(\delta)}(x)=E^{\alpha.\beta.}_x\bigg[\int_0^{\tau(\delta)}(f^{\alpha_t\beta_t})^{(\delta)}\big(x_t(\delta)\big)e^{-\phi_t}dt+g^{(\delta)}\big(x_{\tau(\delta)}(\delta)\big)e^{-\phi_{\tau(\delta)}}\bigg].
\end{equation}

\begin{lemma}\label{lemma1} 
For each $\delta\in(0,1)$, $v^{(\delta)}$ is in the class of $C^{1,\gamma}(\bar D)$ for some positive constant $\gamma$ depending only on $d$, $K_0$ and $\delta$, and we have 
\begin{equation}\label{conver}
\lim_{\delta\downarrow0}|v^{(\delta)}-v^+|_{0,D}=0.
\end{equation}
\end{lemma}

\begin{proof}
For the controlled diffusion process given by (\ref{diffusione}), its diffusion term is of size $d\times(d+d_1)$ in the form of $(\sigma^{\alpha\beta})^{(\delta)}=(\sigma^{\alpha\beta}|\delta I)$. As a result, the associated Dirichlet problem for the Isaacs equation of $v^{(\delta)}$ is
\begin{equation}\tag{{\bf{I}{$^{\delta}$}}}\label{idel}
\left\{
\begin{array}{rcll}
\displaystyle\inf_{\beta\in B}\sup_{\alpha\in A}\big[\tr\big((a^{\alpha\beta})^{(\delta)} u_{xx}\big)+(b^{\alpha\beta})\cdot u_x-c^{\alpha\beta} u+(f^{\alpha\beta})^{(\delta)}\big]&=&0&\mbox{in }D\\
u&=&g^{(\delta)}&\mbox{on }\partial D,
\end{array}
\right.
\end{equation}
where $(a^{\alpha\beta})^{(\delta)}=a^{\alpha\beta}+(\delta^2/2) I$, which is nondegenerate for each $\delta>0$. By Theorem 4.3 in \cite{MR2500878}, $v^{(\delta)}\in C^0(\bar D)$ is a viscosity solution to (\ref{idel}). Then by Theorem 2.3 in \cite{MR1034037}, $v^{(\delta)}\in C^{1,\gamma}(\bar D)$ for some positive constant $\gamma$ depending only on $d$, $K_0$, $\delta$.

We next prove the convergence result (\ref{conver}). Notice that $g^{(\delta)}(\cdot)$ and $e^{-x}$ are globally Lipschitz in $D$ and $\mathbb R^+$, respectly, $(f^{\alpha\beta})^{(\delta)}(\cdot)$ is globally Lipschitz in $D$, uniformly with respect to $(\alpha,\beta)\in A\times B$, and
\begin{align*}
|v^{(\delta)}-v^+|_{0,D}\le&\sup_{x\in D}\sfm\sup_{\beta.\in\mathfrak B}|(v^{\boldsymbol\mu(\beta.)\beta.})^{(\delta)}(x)-v^{\boldsymbol\mu(\beta.)\beta.}(x)|\\
\le&\sup_{x\in D}\sfa\sup_{\beta.\in\mathfrak B}|(v^{\alpha.\beta.})^{(\delta)}(x)-v^{\alpha.\beta.}(x)|.
\end{align*}
Therefore, to show (\ref{conver}), it suffices to prove that
\begin{align}
&\lim_{\delta\downarrow0}\sup_{x\in D}\sfa \sup_{\beta.\in\mathfrak B}E^{\alpha.\beta.}_x\sup_{t\le\tau(\delta)\wedge\tau}|x_t(\delta)-x_t|=0,\label{convx}\\
&\lim_{\delta\downarrow0}\sup_{x\in D}\sfa \sup_{\beta.\in\mathfrak B}E^{\alpha.\beta.}_x|\tau(\delta)\vee\tau-\tau(\delta)\wedge\tau|=0.\label{convt}
\end{align}

To prove (\ref{convx}), we notice that, for any constant $T\in[1,\infty)$,
\begin{align*}
\MoveEqLeft E^{\alpha.\beta.}_x\sup_{t\le\tau(\delta)\wedge\tau}|x_t(\epsilon)-x_t|\\
&\le E^{\alpha.\beta.}_x\sup_{t\le\tau(\delta)\wedge\tau\wedge T}|x_t(\delta)-x_t|+KP^{\alpha.\beta.}_x(\tau>T)\\
&=\delta E^{\alpha.\beta.}_x\sup_{t\le\tau(\delta)\wedge\tau\wedge T}|\tilde w_t|+\frac{K}{T}E^{\alpha.\beta.}_x\tau\\
&\le3\delta T+\frac{K}{T}|\psi|_{0,D}.
\end{align*}
By taking the supremum with respect to $x$, $\alpha.$ and $\beta.$  on the left side and letting first $\delta\downarrow0$ and then $T\uparrow\infty$, we obtain (\ref{convx}).

To prove (\ref{convx}), we notice that
$$E^{\alpha.\beta.}_x|\tau(\delta)\vee\tau-\tau(\delta)\wedge\tau|=E^{\alpha.\beta.}_x(\tau-\tau(\delta))\mathbbm1_{\tau\ge\tau(\delta)}+E^{\alpha.\beta.}_x(\tau(\delta)-\tau)\mathbbm1_{\tau<\tau(\delta)}.$$
Then we estimate both terms. We have
\allowdisplaybreaks\begin{align*}
\MoveEqLeft E^{\alpha.\beta.}_x(\tau-\tau(\delta))\mathbbm1_{\tau>\tau(\delta)}\\
&\le-E_x^{\alpha.\beta.}\int_{\tau(\delta)\wedge\tau}^{\tau}L\psi(x_t)dt\\
&=E^{\alpha.\beta.}_x\psi(x_{\tau(\delta)})\mathbbm1_{\tau(\delta)<\tau}\\
&=E^{\alpha.\beta.}_x\Big(\psi (x_{\tau(\delta)})-\psi(x_{\tau(\delta)}(\delta))\Big)\mathbbm1_{\tau(\delta)<\tau}\\
&\le E^{\alpha.\beta.}_x\Big(\psi (x_{\tau(\delta)})-\psi(x_{\tau(\delta)}(\delta))\Big)\mathbbm1_{\tau(\delta)<\tau\le T}+2|\psi|_{0,D}P_x^{\alpha.\beta.}(\tau>T)\\
&\le|\psi|_{0,1,D}E^{\alpha.\beta.}_x\sup_{t\le \tau(\delta)\wedge\tau\wedge T}|x_t-x_t(\delta)|+\frac{K}{T}E_x^{\alpha.\beta.}\tau.
\end{align*}
Similarly, by notice that for sufficiently small $\delta$,
\begin{equation}\label{Lpsie}
(L^{\alpha\beta})^{(\delta)} \psi:=L^{\alpha\beta}\psi+\frac{\delta^2}{2}\Delta \psi\le-1/2,
\end{equation}
we have
\begin{align*}
\MoveEqLeft E_x^{\alpha.\beta.}(\tau(\delta)-\tau)\mathbbm1_{\tau<\tau(\delta)}\\
&\le-2E_x^{\alpha.\beta.}\int_{\tau(\delta)\wedge\tau}^{\tau(\delta)}L^{(\delta)}\psi(x_t(\delta))dt\\
&= 2E_x^{\alpha.\beta.}\Big(\psi(x_\tau(\delta))-\psi(x_\tau)\Big)\mathbbm1_{\tau<\tau(\delta)}\\
&\le2|\psi|_{0,1,D}E_x^{\alpha.\beta.}\sup_{t\le \tau(\delta)\wedge\tau\wedge T}|x_t-x_t(\delta)|+\frac{K}{T}E_x^{\alpha.\beta.}\tau.
\end{align*}
It turns out that
$$E^{\alpha.\beta.}_x|\tau(\delta)\vee\tau-\tau(\delta)\wedge\tau|\le \delta KT+\frac{K}{T}.$$
Again, by taking the supremum over $D$, $\mathfrak A$, $\mathfrak B$ on the left side of the inequality and letting first $\delta\downarrow0$ and then $T\uparrow\infty$, we obtain (\ref{convt}).
\end{proof}

\begin{remark}\label{remark1}
Due to Lemma \ref{lemma1}, to prove that the upper value function $v^+$ is locally Lipschitz in $D$, it suffices to obtain first derivative estimates on $v^{(\delta)}$ which are independent of $\delta$. See more details in the proof of Theorem \ref{thm1}.
\end{remark}

\subsection{Quasiderivatives and their properties}
In this subsection, we define the quasiderivatives and supermartingales. 

We start from introducing auxiliary functions. For constants $\kappa$ and $\lambda$ satisfying $0<\kappa<\lambda^2<\lambda<1$, let
\begin{align*}
D_\kappa=&\{x\in D:\kappa<\psi(x)\},\\
D^\lambda=&\{x\in D: \psi(x)<\lambda\},\\
D_\kappa^\lambda=&\{x\in D: \kappa<\psi(x)<\lambda\}.
\end{align*}
Note that $\lambda$ is constant throughout this article, and $\kappa$ is an arbitrary constant in the interval $(0,\lambda^2)$ in this subsection, which will approach zero in the proof of Theorems \ref{thm1}. The boundary auxiliary functions $\B_1(x,\xi)$, $r_1(x,\xi)$ and  $P_1(x,\xi)$ are defined from $D^\lambda\times\Rd$ to $\mathbb R^+$, $\mathbb R$ and $\operatorname{Skew}(d,\mathbb R)$, respectively, where $\operatorname{Skew}(d,\mathbb R)$ denotes the set of all $d\times d$-size skew-symmetric matrices, and $\pi_1^{\alpha\beta}(x,\xi)$ is from $A\times B\times D^\lambda\times\Rd$ to $\mathbb R^{d_1}$. The interior auxiliary functions $\B_2(x,\xi)$, $r_2(x,\xi)$ and  $P_2(x,\xi)$ are defined from  $D_{\lambda^2}\times\Rd$ to $\mathbb R^+$, $\mathbb R$, $\operatorname{Skew}(d,\mathbb R)$, respectively, and $\pi^{\alpha\beta}_2(x,\xi)$ is from $A\times B\times D_{\lambda^2}\times\Rd$ to $\mathbb R^{d_1}$. Precisely: 
\begin{itemize}
\item
In $D^\lambda\times\Rd$, we define
\begin{align*}
&\mathrm{B}_1(x,\xi)=\gamma_2\bigg[\gamma_1|\xi|^2+\pxsop\bigg],\\
&r_1(x,\xi)=-\frac{1}{|\psi_x|^2}\sum_{k=1}^d\psi_{x^k}(\psi_{x^k})_{(\xi)}+\frac{\psi_{(\xi)}}{\psi},\\
&P_1^{jk}(x,\xi)=\frac{(\psi_{x^j})_{(\xi)}\psi_{x^k}-(\psi_{x^k})_{(\xi)}\psi_{x^j}}{|\psi_x|^2},\\
&(\pi_1^{\alpha\beta})^k(x,\xi)=\frac{1}{2\gamma_2}\Big(1-\frac{\psi}{2\lambda}\Big)\bigg[\frac{\psi_{(\xi)}}{\psi}\psi_{(\sigma_k^{\alpha\beta})}+\gamma_1(\xi,\sigma_k^{\alpha\beta})\bigg],
\end{align*}
with
$$\gamma_1=1+\frac{1}{8\lambda}\psi\Big(1-\frac{1}{4\lambda}\psi\Big),\quad\gamma_2=\lambda^2+\psi\Big(1-\frac{1}{4\lambda}\psi\Big).$$
\item
In $D_{\lambda^2}\times\Rd$, we define
\allowdisplaybreaks\begin{align*}
&\mathrm{B}_2(x,\xi)=\lambda^{3\theta}\psi^{1-2\theta}\bigg[K_1|\xi|^2+\pxsop\bigg],\\
&r_2(x,\xi)=\theta\frac{\psi_{(\xi)}}{\psi},\\
&P_2^{jk}(x,\xi)=0,\\
&(\pi^{\alpha\beta}_2)^k(x,\xi)=\frac{\theta(1-2\theta)^2}{2(1-3\theta)\psi^2}\bigg[K_1\psi(\xi,\sigma^{\alpha\beta}_k)+\psi_{(\xi)}\psi_{(\sigma^{\alpha\beta}_k)}\bigg],
\end{align*}
where $\theta\in(0,1/3)$ and $K_1\in[1,\infty)$ are constants depending on $K_0$, $d$, $d_1$, $D$. 
\item
In $D\times\Rd$, for convenience of notation, we define
\begin{align*}
&\overline{\mathrm B}(x,\xi)=\mathbbm1_{x\in D^\lambda}\B_1(x,\xi)+\mathbbm1_{x\in \bar D_{\lambda^2}}\B_2(x,\xi),\\
&\underline\B(x,\xi)=\left\{
\begin{array}{ll}
\B_1(x,\xi)&\mbox{ in }D^{\lambda^2}\\
\B_1(x,\xi)\wedge\B_2(x,\xi)&\mbox{ in }\bar D_{\lambda^2}^\lambda\\
\B_2(x,\xi)&\mbox{ in }D_{\lambda}.
\end{array}
\right.
\end{align*}
\end{itemize}

The idea of introducing the auxiliary function $r(x,\xi)$ and $\pi^{\alpha\beta}(x,\xi)$ are due to changing the random time and probability space, respectively, from probability theory. The reason of introducing $P(x,\xi)$ is to make use of the invariance property in Assumption \ref{invc}. The auxiliary functions $\B_1(x,\xi)$ and $\B_2(x,\xi)$ will play the roles of boundary barrier and interior barrier, respectively, since based on our definition of the state process $x_t^{\alpha.\beta.,x}$ and its quasiderivative $\xi_t^{\alpha.\beta.,\xi}$ (cf. Equation (\ref{itoxi})), the processes $\B_1(x_t^{\alpha.\beta.,x},\xi_t^{\alpha.\beta,\xi})$ (resp. $\B_2(x_t^{\alpha.\beta.,x},\xi_t^{\alpha.\beta,\xi})$) is a local supermartingales when the state process $x_t^{\alpha.,\beta.,x}$ is in $D_\kappa^\lambda$ (resp. $D_{\lambda^2}$).

Now we construct the (first order) quasiderivatives $\xi_t^{\alpha.\beta.,\xi}$ with the help of the auxiliary functions $r,P,\pi$. Recall that for each $x\in D$, $\alpha.\in\mathfrak A$, $\beta.\in\mathfrak B$, the controlled diffusion process $x_t^{\alpha.\beta.,x}$ is defined by (\ref{itox}). Define the following stopping times of $x_t^{\alpha.\beta.,x}$:
\begin{align*}
\tau^{\alpha.\beta.,x}_\kappa=&\inf\{t\ge0: x_t^{\alpha.\beta.,x}\notin D_\kappa\},\\
\tau^{\alpha.\beta.,x}_{-1}=&0,\\
\tau^{\alpha.\beta.,x}_0=&\inf\{t\ge0:\psi(x_t^{\alpha.\beta.,x})\le \lambda^2\},\\
\tau^{\alpha.\beta.,x}_1=&\tau_\kappa^{\alpha.\beta.,x}\wedge\inf\{t\ge\tau^{\alpha.\beta.,x}_0:\psi(x_t^{\alpha.\beta.,x})\ge \lambda\},
\end{align*}
and recursively, for each $n\in\mathbb N$,
\begin{align*}
\tau^{\alpha.\beta.,x}_{2n}=&\tau_\kappa^{\alpha.\beta.,x}\wedge\inf\{t\ge\tau_{2n-1}^{\alpha.\beta.,x}:\psi(x_t^{\alpha.\beta.,x})\le \lambda^2\},\\
\tau^{\alpha.\beta.,x}_{2n+1}=&\tau_\kappa^{\alpha.\beta.,x}\wedge\inf\{t\ge\tau_{2n}^{\alpha.\beta.,x}:\psi(x_t^{\alpha.\beta.,x})\ge \lambda\}.
\end{align*}

For each $\epsilon\in[0,1]$, $(\alpha.,\beta.)\in\mathfrak A\times\mathfrak B$, $x, y\in  D_\kappa$, $\xi\in \Rd$, we consider the following $(3d+1)$-dimensional system of stochastic differential equations on $[0,\tau_\kappa^{\alpha.\beta.,x})$:
\allowdisplaybreaks\begin{align}
\label{itox}
x_t^{\alpha.\beta.,x}=&x+\int_0^t\sigma^{\alpha_s\beta_s}dw_s+\int_0^tb^{\alpha_s\beta_s}ds,\\
\label{itoxi}
\xi_t^{\alpha.\beta.,\xi}=&\xi+\int_0^t\Big[r^{\alpha.\beta.}_s\sigma^{\alpha_s\beta_s}+P^{\alpha.\beta.}_s\sigma^{\alpha_s\beta_s} \Big]dw_s\\
&+\int_0^t\Big[2r^{\alpha.\beta.}_sb^{\alpha_s\beta_s}-\sigma^{\alpha_s\beta_s} \pi^{\alpha.\beta.}_s\Big]ds,\nonumber\\
\label{itoxia}
\zeta_t^{\alpha.\beta.,\zeta}=&\zeta+\int_0^t\pi^{\alpha.\beta.}_sdw_s,\\
\label{itoy}
y_t^{\alpha.\beta.,y}( \epsilon)=&y+\int_0^t\sqrt{\theta_s^{\alpha.\beta.}(\epsilon)}Q_s^{\alpha.\beta.}(\epsilon)\sigma^{\alpha_s\beta_s}dw_s\\
&+\int_0^t\Big[\theta_s^{\alpha.\beta.}(\epsilon) b^{\alpha_s\beta_s}-\sqrt{\theta_s^{\alpha.\beta.}(\epsilon)}Q_s^{\alpha.\beta.}(\epsilon)\sigma^{\alpha_s\beta_s} \epsilon\pi_s^{\alpha.\beta.}\Big]ds,\nonumber
\end{align}
where for each $\alpha.\in\mathfrak A$, $\beta.\in\mathfrak B$ and $s\in [0,\tau_\kappa^{\alpha.\beta.,x})$, 
\begin{align*}
\MoveEqLeft(r_s^{\alpha.\beta.},P_s^{\alpha.\beta.},\pi_s^{\alpha.\beta.})\\
\nonumber&=\left\{
\begin{array}{rl}
(r_1,P_1,\pi^{\alpha_s\beta_s}_1)(x_s^{\alpha.\beta.,x},\xi_s^{\alpha.\beta.,\xi})&\mbox{when $s\in[\tau_{2n-2}^{\alpha.\beta.,x},\tau_{2n-1}^{\alpha.\beta.,x})$}\\
(r_2,P_2,\pi^{\alpha_s\beta_s}_2)(x_s^{\alpha.\beta.,x},\xi_s^{\alpha.\beta.,\xi})&\mbox{when $s\in[\tau_{2n-1}^{\alpha.\beta.,x},\tau_{2n}^{\alpha.\beta.,x})$},
\end{array}
\right.
\end{align*}
and
\begin{align}
\theta_s^{\alpha.\beta.}(\epsilon)&=1+\frac{1}{\pi}\arctan\big(\pi2\epsilon r^{\alpha.\beta.}_s\big),\\
Q_s^{\alpha.\beta.}(\epsilon)&=\exp({ \epsilon P_s^{\alpha.\beta.}}).\label{Qs}
\end{align}
It is worth pointing out that among (\ref{itox}) - (\ref{itoy}), only (\ref{itoxi}) is a system of stochastic differential equations, the other three are not.

Due to the definition of the auxiliary functions $r(x,\xi)$, $P(x,\xi)$, $\pi^{\alpha\beta}(x,\xi)$, the system of stochastic differential equations (\ref{itox}) - (\ref{itoy}) is uniquely solvable (cf. Lemma 6.1 in \cite{InteriorRegularityI}). Notice that $y_t^{\alpha.\beta.,y}(0)=x_t^{\alpha.\beta.,y}$, so $y_t^{\alpha.\beta.,y}(\epsilon)$ can be regarded as a perturbation of $x_t^{\alpha.\beta.,x}$. The process $\xi_t^{\alpha.\beta.,\xi}$ is called a first quasiderivative of $x_t^{\alpha.\beta.,x}$ at $x$ in the direction $\xi$ because (\ref{itoxi}) can be obtained by formally differentiate $y_t^{\alpha.\beta.,x+\epsilon\xi}(\epsilon)$ with respect to $\epsilon$ at $\epsilon=0$. And $\zeta_t^{\alpha.\beta.,\zeta}$ is called an adjoint process of $\xi_t^{\alpha.\beta.,\xi}$.

In the following lemma, we collect some properties of the first quasiderivatives $\xi_t^{\alpha.\beta.,\xi}$, all of which can be found in Lemma 7.4 in \cite{InteriorRegularityI}, by replacing $\sup_{\alpha.\in\mathfrak A}$ with $\sup_{(\alpha.,\beta.)\in\mathfrak A\times\mathfrak B}$.

\begin{lemma}\label{lemma2}
Fix any $x\in D_\kappa$ and $\xi\in\Rd$. We have the following estimates.
\begin{enumerate}
\item[(a1)] $\displaystyle{\sfa\sfb E^{\alpha.\beta.}_{x,\xi}\sup_{t\le\tau_\kappa}|\xi_t|^2\le N\overline{\mathrm B}(x,\xi)}$;
\item[(a2)] $\displaystyle{\sfa\sfb E^{\alpha.\beta.}_{x,\xi,0}\sup_{t\le\tau_\kappa}|\zeta_t|^2\le N\overline{\mathrm B}(x,\xi)}$;
\item[(a3)] $\displaystyle{\sfa\sfb E^{\alpha.\beta.}_{x,\xi}\int_0^{\tau_\kappa}\Big(|\xi_t|^2+\pxtsops\Big) dt\le N\overline{\mathrm B}(x,\xi)}$;
\item[(a4)] $\displaystyle{\sfa\sfb E^{\alpha.\beta.}_{x,\xi}\underline{\mathrm B}(x_\gamma,\xi_\gamma)\le 2\overline{\mathrm B}(x,\xi)}$, $\forall\gamma^{\alpha.\beta.,x}\le\tau_\kappa^{\alpha.\beta.,x}$;
\end{enumerate}
where $N$ is a constant depending only on $K_0,d,d_1$ and $\lambda$.

For any constants $T\in[1,\infty)$ and $p\in(0,\infty)$, we have
\begin{enumerate}
\item[(b1)] $\displaystyle\sfa\sfb E_{x,\xi}^{\alpha.\beta.}\sup_{t\le\tau_\kappa\wedge T}|\xi_t|^p<\infty$.
\end{enumerate}

Furthermore, introduce constants $q\in[0,p)$ and $\epsilon_0\in(0,1]$ satisfying $\{y:|y-x|\le\epsilon_0|\xi|\}\subset D_\kappa$. For each $\epsilon\in(0,\epsilon_0]$, define
\begin{align*}
\bar\tau_\kappa^{\alpha.\beta.,x+\epsilon\xi}=&\inf\{t\ge0: y_t^{\alpha.\beta.,x+\epsilon\xi}\notin D_\kappa\},\\
\iota_\kappa^{\alpha.\beta.,x,\xi}(\epsilon)=&\tau_\kappa^{\alpha.\beta.,x}\wedge\bar\tau_\kappa^{\alpha.\beta.,x+\epsilon\xi}.
\end{align*}
Let
$h:  A\times B\times \bar D_\kappa\rightarrow \mathbb R; (\alpha,\beta,x)\mapsto h^{\alpha\beta}(x)$ 
be a Borel function which is (globally) Lipschitz with respect $x$ in $\bar D_\kappa$, uniformly with respect to $(\alpha,\beta)$,  i.e. $\sup_{(\alpha,\beta)\in A\times B}|h^{\alpha\beta}(\cdot)|_{0,1,D_\kappa}<\infty$. We have the following convergence results.

\begin{enumerate}
\item[(c1)] $\displaystyle\lim_{\epsilon\downarrow0}\sfa\sfb E^{\alpha.\beta.}_{x,\xi}\sup_{t\le\iota_\kappa(\epsilon)\wedge T}\frac{|y_t^{\alpha.\beta.,x+\epsilon\xi}(\epsilon)-x_t^{\alpha.\beta.,x}|^p}{\epsilon^{q}}=0$;
\item[(c2)] $\displaystyle\lim_{\epsilon\downarrow0}\sfa\sfb E^{\alpha.\beta.}_{x,\xi}\sup_{t\le\iota_\kappa(\epsilon)\wedge T}\Big|\frac{y_t^{\alpha.\beta.,x+\epsilon\xi}(\epsilon)-x_t^{\alpha.\beta.,x}}{\epsilon}-\xi_t^{\alpha.\beta.,\xi}\Big|^p=0$;
\item[(c3)] $\displaystyle\lim_{\epsilon\downarrow0}\sfa\sfb E^{\alpha.\beta.}_{x,\xi}\sup_{t\le\iota_\kappa(\epsilon)\wedge T}\frac{|h^{\alpha_t\beta_t}(y_t(\epsilon))-h^{\alpha_t\beta_t}(x_t)|^p}{\epsilon^{q}}=0$;
\item[(c4)] $\displaystyle\lim_{\epsilon\downarrow0}\sfa\sfb E^{\alpha.\beta.}_{x,\xi}\sup_{t\le\iota_\kappa(\epsilon)\wedge T}\bigg|\frac{h^{\alpha_t\beta_t}(y_t(\epsilon))-h^{\alpha_t\beta_t}(x_t)}{\epsilon}-h^{\alpha_t\beta_t}_{(\xi_t)}(x_t)\bigg|^p=0,$ if furthermore, for all $(\alpha,\beta)\in(A\times B)$, $h^{\alpha\beta}(\cdot)\in C^1(\bar D_\kappa)$, and $h_x^{\alpha\beta}(\cdot)$ are continuous in $x$, uniformly with respect to $(\alpha,\beta)$;
\end{enumerate}

\end{lemma}
\begin{proof}
It suffices to go through Sections 6 and 7 in \cite{InteriorRegularityI} by replacing $\mathfrak A$ with $\mathfrak A\times\mathfrak B$.
\end{proof}

\subsection{Stochastic representations and dynamic programming principle} In this subsection, $\epsilon$ is fixed.

\begin{lemma}\label{lemma3} Recall that $v^+$ is defined by (\ref{payoff}) and (\ref{uppervalue}). Also assume that there is a constant $\delta\in(0,1]$ such that for all $\alpha\in A$, $\beta\in B$, $\eta\in\Rd$, we have
$$\delta|\eta|^2\le a^{\alpha\beta}_{ij}\eta^i\eta^j,$$
i.e. the diffusion matrices are uniformly non-degenerate.
\begin{enumerate}
\item[(i)] Let $\varkappa^{\alpha.\beta.,x}$ be an $\{\mathcal F_t\}$-stopping time defined for each $\alpha.\in\mathfrak A$, $\beta\in\mathfrak B$, $x\in D$, such that $\varkappa^{\alpha.\beta.,x}\le \tau^{\alpha.\beta.,x}$. Then the dynamic programming principle with random stopping times holds, i.e. for any $x\in D$,
\begin{equation}\label{dpp}
v^+(x)=\sfm\ifb E_x^{\boldsymbol\mu(\beta.)\beta.}\Big[v^+(x_\varkappa)e^{-\phi_\varkappa}+\int_0^\varkappa f^{\alpha_t\beta_t}(x_t)e^{-\phi_t}dt\Big].
\end{equation}

\item[(ii)] Let $\mathfrak Q$ be the set of all progressively-measurable processes $q_t$ with value in $\mathbb O_d$ for all $t\ge0$, and for each $z\in D$, $\alpha.\in\mathfrak A$, $\beta.\in\mathfrak B$, $q.\in\mathfrak Q$, define
\begin{align}
\label{itoz}z_t^{\alpha.\beta.q.,z}( \epsilon)=&z+\int_0^t\sqrt{\theta_s^{\alpha.\beta.}(\epsilon)}q_s\sigma^{\alpha_s\beta_s}dw_s\\
&+\int_0^t\Big[\theta_s^{\alpha.\beta.}(\epsilon) b^{\alpha_s\beta_s}-\sqrt{\theta_s^{\alpha.\beta.}(\epsilon)}q_s\sigma^{\alpha_s\beta_s} \epsilon\pi_s^{\alpha.\beta.}\Big]ds,\nonumber
\end{align}
and
$$\hat\tau^{\alpha.\beta.q.,z}(\epsilon)=\inf\{t\ge0: z_t^{\alpha.\beta.q.,z}\notin D\}.$$

Let $\gamma^{\alpha.\beta.q.,z}$ be an $\{\mathcal F_t\}$-stopping time defined for each $\alpha.\in\mathfrak A$, $\beta\in\mathfrak B$, $q.\in\mathfrak Q$, $z\in D$, such that $\gamma^{\alpha.\beta.q.,z}\le \hat\tau^{\alpha.\beta.q.,z}(\epsilon)$.

If Assumption \ref{invc} holds by satisfying Condition (\ref{othinv}), then for all $z\in D$, 
\begin{align}\label{dynamic}
v^+(z)=&\sup_{\boldsymbol\rho\in\mathcal P} \inf_{(\beta.,q.)\in\mathfrak B\times\mathfrak Q}E_{z}^{\boldsymbol\rho(\beta.,q.)\beta.q.}\Big[v^+(z_\gamma(\epsilon))p_\gamma(\epsilon)e^{-\phi_\gamma(\epsilon)}\\
&+\int_0^\gamma\theta_t(\epsilon)f^{\alpha_t\beta_t}(z_t(\epsilon))p_t(\epsilon)e^{-\phi_t(\epsilon)}dt\Big],\nonumber
\end{align}
where 
\begin{align}
\label{phit}
\phi_t^{\alpha.\beta.}( \epsilon)=&\int_0^t\theta_s^{\alpha.\beta.}(\epsilon)c^{\alpha_s\beta_s}ds,\\
\label{pt}
p_t^{\alpha.\beta.}( \epsilon)=&\exp\bigg(\int_0^t \epsilon\pi^{\alpha.\beta.}_sdw_s-\frac{1}{2}\int_0^t| \epsilon\pi^{\alpha.\beta.}_s|^2ds\bigg).
\end{align}
If Assumption \ref{invc} holds by satisfying Condition (\ref{othinv2}), we have an equation analogous to (\ref{dynamic}). 
\end{enumerate}
\end{lemma}
\begin{proof}
Equation (\ref{dpp}) is essentially the same as Equation (2.6) in \cite{Krylov:2012fk} for $\lambda_t^{\alpha.\beta.,x}\equiv0$. Therefore to prove (i), it suffices to observe that our Assumption \ref{geod} implies Assumption 2.2 in \cite{Krylov:2012fk}. Consequently, Theorem 2.1 in \cite{Krylov:2012fk} is applicable in our situation.

In the proof of (ii), we first make use of the invariance property, i.e. Assumption \ref{invc}, to involve $q_s$ in the original stochastic representation of $v^+$. Then we apply Theorems 2.3 in \cite{Krylov:2012fk} to deal with $\theta_t^{\alpha.\beta.}$ coming from random time change. We last apply Girsanov's Theorem to add $\pi_t^{\alpha.\beta.}$ into the picture.

First, due to the uniform non-degeneracy assumption in this lemma, $v^+$ is the unique viscosity solution of $H^+(u_{xx}, u_x, u, x)=0$, so by Condition (\ref{othinv}) in Assumption 2.2, it is also the unique viscosity solution of 
\begin{equation}\label{isaacsq}
\inf_{(\beta,q)\in B\times\mathbb O_d}\sup_{\alpha\in A}H^{\alpha\beta q}( u_{xx}, u_x, u, x)=0.
\end{equation}
On the other hand, since the elliptic equation (\ref{isaacsq}) is still uniform non-degenerate, we know that its unique viscosity solution (under the same boundary condition) has the stochastic representation
$$v^+(x)=\sup_{\boldsymbol\rho\in\mathcal P} \inf_{(\beta.,q.)\in\mathfrak B\times\mathfrak Q}E_{z}^{\boldsymbol\rho(\beta.,q.)\beta.q.}\bigg[\int_0^{\tilde\tau} f^{\alpha_t\beta_t}(\tilde x_t)e^{-\phi_t}dt+g(\tilde x_{\tilde\tau})e^{-\phi_{\tilde\tau}}\bigg],$$
where
$$\tilde x_t^{\alpha.\beta.q.,x}=x+\int_0^t q_s\sigma^{\alpha_s\beta_s}dw_s+\int_0^t b^{\alpha_s\beta_s}ds.$$

We proceed by considering the set 
$$\{(\theta,x):\theta\in\mathbb R^+, x\in D)\},$$
where $\theta$ will play the role of $p$  in  \cite{Krylov:2012fk}. Following the setup after Theorem 2.2 in \cite{Krylov:2012fk}, we define
\begin{align*}
&\check \sigma^{\alpha\beta q}(\theta,x)=\sqrt\theta (q\sigma^{\alpha\beta}),\\
&\check a^{\alpha\beta q}(\theta,x)=(1/2)\check \sigma^{\alpha\beta q}(\theta,x)(\check \sigma^{\alpha\beta q}(\theta, x))^*,\\
&(\check b^{\alpha\beta q},\check c^{\alpha\beta q},\check f^{\alpha\beta q})(\theta, x)=\theta (b^{\alpha\beta}, c^{\alpha\beta},f^{\alpha\beta}(x)),\\
&\check r^{\alpha\beta q}(\theta,x)=\theta, \quad\bar\theta=1,
\end{align*}
where $(\beta,q)\in B\times\mathbb O_d$ here plays the role of $\beta\in B$ in \cite{Krylov:2012fk}. Notice that
$$(\check a^{\alpha\beta q},\check b^{\alpha\beta q},\check c^{\alpha\beta q},\check f^{\alpha\beta q})(\theta,x)=\theta(qa^{\alpha\beta}q^*,b^{\alpha\beta},c^{\alpha\beta},f^{\alpha\beta}(x)).$$
By applying Theorem 2.3 in \cite{Krylov:2012fk} for $\lambda_t^{\alpha.\beta.}\equiv0$, $\Pi=\operatorname{id}$, $F_t=t$ and $\psi\equiv1$, 
we obtain that
\begin{align*}
v^+(z)=&\sup_{\boldsymbol\rho\in\mathcal P} \inf_{(\beta.,q.)\in\mathfrak B\times\mathfrak Q}E_{z}^{\boldsymbol\rho(\beta.,q.)\beta.q.}\bigg[v^+(\check x_{\check\gamma}(\epsilon))e^{-\phi_{\check\gamma}(\epsilon)}\\
&+\int_0^{\check\gamma}\theta_t(\epsilon)f^{\alpha_t\beta_t}(\check x_t(\epsilon))e^{-\phi_t(\epsilon)}dt\bigg],
\end{align*}
with
\begin{align}\label{checkx}
\check x_t^{\alpha.\beta.q.,\check x}( \epsilon)=\check x+\int_0^t\sqrt{\theta_s^{\alpha.\beta.}(\epsilon)}q_s\sigma^{\alpha_s\beta_s}dw_s+\int_0^t\theta_s^{\alpha.\beta.}(\epsilon) b^{\alpha_s\beta_s}ds,
\end{align}
which is exactly (\ref{dynamic}) when $\pi_t^{\alpha.\beta.}\equiv0$.

Last, by Remark 2.4 in \cite{Krylov:2012fk}, the function $v^+(\cdot)$ is uniquely defined once the functions $\sigma, b,c, f, g$ are given, so it does not depend on the probability space and the Wiener process. Therefore by Lemma 2.1 in \cite{MR637615} with minor modifications we obtain (\ref{dynamic}). More precisely, let
\begin{align}
\label{321}u^{\alpha.\beta.q.}(z;\gamma)
=&E_{z}^{\alpha.\beta.q.}\bigg[v^+(\check x_{\gamma}(\epsilon))e^{-\phi_{\gamma}(\epsilon)}+\int_0^{\gamma}\theta_t(\epsilon)f^{\alpha_t\beta_t}(\check x_t(\epsilon))e^{-\phi_t(\epsilon)}dt\bigg],\\
w^{\alpha.\beta.q.}(z;\gamma)
=&E_{z}^{\alpha.\beta.q.}\bigg[v^+(z_\gamma(\epsilon))p_\gamma(\epsilon)e^{-\phi_\gamma(\epsilon)}+\int_0^\gamma\theta_t(\epsilon)f^{\alpha_t\beta_t}(z_t(\epsilon))p_t(\epsilon)e^{-\phi_t(\epsilon)}dt\bigg].
\end{align}
By Girsanov's Theorem, for each $(\alpha.,\beta.,q.)\in\mathfrak A\times\mathfrak B\times\mathfrak Q$, $\gamma=\gamma^{\alpha.\beta.q.,z}\le\hat\tau^{\alpha.\beta.q.,z}(\epsilon)$, $w^{\alpha.\beta.q.}(z;\gamma)$ can be represented as $u^{\alpha.\beta.q.}(z;\gamma)$, if in the definition of $u^{\alpha.\beta.q.}(z;\gamma)$, (\ref{321}), $\check x_t^{\alpha.\beta.q.,\check x}(\epsilon)$ is replaced by $z_t^{\alpha.\beta.q.,z}(\epsilon)$, with the Wiener process $\bar w_t$ and the probability space $\bar P$, where
$$\bar P(d\omega)=p^{\alpha.\beta.}_\gamma(\epsilon)P(d\omega),$$
$$\bar w_t=w_t-\epsilon\int_0^{t\wedge\gamma}\pi_s^{\alpha.\beta.}ds.$$
By Theorem 2.3 in \cite{Krylov:2012fk} for $\lambda_t^{\alpha.\beta.}\equiv0$, $\Pi=\operatorname{id}$, $F_t=t$ and $\psi\equiv1$, 
$$\sup_{\boldsymbol\rho\in\mathcal P} \inf_{(\beta.,q.)\in\mathfrak B\times\mathfrak Q}u^{\boldsymbol\rho(\beta.,q.)\beta.q.}(z;\gamma^{\boldsymbol\rho(\beta.,q.)\beta.q.})$$ 
does not depend on the probability space or the Wiener process. As a result, there exists a $\boldsymbol\rho_0\in\mathcal P$, such that for any $\epsilon>0$, $(\beta.,q.)\in\mathfrak B\times\mathfrak Q$, we have
\begin{align*}
v^+(z)=&\sup_{\boldsymbol\rho\in\mathcal P} \inf_{(\beta.,q.)\in\mathfrak B\times\mathfrak Q}u^{\boldsymbol\rho(\beta.,q.)\beta.q.}(z;\gamma^{\boldsymbol\rho(\beta.,q.)\beta.q.})\\
\le&\inf_{(\beta.,q.)\in\mathfrak B\times\mathfrak Q}u^{\boldsymbol\rho_0(\beta.,q.)\beta.q.}(z;\gamma^{\boldsymbol\rho_0(\beta.,q.)\beta.q.})+\epsilon\\
\le&\inf_{(\beta.,q.)\in\mathfrak B\times\mathfrak Q}w^{\boldsymbol\rho_0(\beta.,q.)\beta.q.}(z;\gamma^{\boldsymbol\rho_0(\beta.,q.)\beta.q.})+\epsilon\\
\le&\sup_{\boldsymbol\rho\in\mathcal P}\inf_{(\beta.,q.)\in\mathfrak B\times\mathfrak Q}w^{\boldsymbol\rho(\beta.,q.)\beta.q.}(z;\gamma^{\boldsymbol\rho(\beta.,q.)\beta.q.})+\epsilon.
\end{align*}
Therefore, the left-hand side of (\ref{dynamic}) doesn't exceed the right-hand side. 

In order to prove the reverse inequality, we add one more equation to the system (\ref{itoz}):
$$y^{\alpha.\beta.q.,y}_t(\epsilon)=y+\epsilon\int_0^t\pi_s^{\alpha.\beta.}y_s^{\alpha.\beta.q.,y}(\epsilon)ds.$$
If we define
\begin{align*}
\MoveEqLeft w^{\alpha.\beta.q.}(z,y;\gamma)\\
=&E_{z,y}^{\alpha.\beta.q.}\bigg[v^+(z_\gamma(\epsilon))y_\gamma(\epsilon)e^{-\phi_\gamma(\epsilon)}+\int_0^\gamma\theta_t(\epsilon)f^{\alpha_t\beta_t}(z_t(\epsilon))y_t(\epsilon)e^{-\phi_t(\epsilon)}dt\bigg],
\end{align*}
then
$$w^{\alpha.\beta.q.}(z;\gamma)=w^{\alpha.\beta.q.}(z,1;\gamma).$$
By Girsanov's Theorem, for each $(\alpha.,\beta.,q.)\in\mathfrak A\times\mathfrak B\times\mathfrak Q$, $\gamma=\gamma^{\alpha.\beta.q.,z}\le\hat\tau^{\alpha.\beta.q.,z}(\epsilon)$, we have
\begin{align}\label{basedon}
u^{\alpha.\beta.q.}(z;\gamma)=w^{\alpha.\beta.q.}(z,1;\gamma),
\end{align}
if $w^{\alpha.\beta.q.}(z,1;\gamma)$ is constructed on the probability space $\bar P$ with Wiener process $\bar w_t$ where
$$\bar P(d\omega)=\exp\bigg(-\int_0^\gamma \epsilon\pi^{\alpha.\beta.}_sdw_s-\frac{1}{2}\int_0^\gamma| \epsilon\pi^{\alpha.\beta.}_s|^2ds\bigg)P(d\omega),$$
$$\bar w_t=w_t+\epsilon\int_0^{t\wedge\gamma}\pi_s^{\alpha.\beta.}ds.$$
Since by Theorem 2.3 in \cite{Krylov:2012fk},
$$\sup_{\boldsymbol\rho\in\mathcal P} \inf_{(\beta.,q.)\in\mathfrak B\times\mathfrak Q}w^{\boldsymbol\rho(\beta.,q.)\beta.q.}(z,1;\gamma^{\boldsymbol\rho(\beta.,q.)\beta.q.})$$ 
does not depend on the probability space or the Wiener process, based on (\ref{basedon}), we have, there exists a $\boldsymbol\rho_0\in\mathcal P$, such that for any $\epsilon>0$, $(\beta.,q.)\in\mathfrak B\times\mathfrak Q$, \begin{align*}
v^+(z)=&\sup_{\boldsymbol\rho\in\mathcal P} \inf_{(\beta.,q.)\in\mathfrak B\times\mathfrak Q}u^{\boldsymbol\rho(\beta.,q.)\beta.q.}(z;\gamma^{\boldsymbol\rho(\beta.,q.)\beta.q.})\\
\ge&\inf_{(\beta.,q.)\in\mathfrak B\times\mathfrak Q}u^{\boldsymbol\rho_0(\beta.,q.)\beta.q.}(z;\gamma^{\boldsymbol\rho_0(\beta.,q.)\beta.q.})\\
\ge&\inf_{(\beta.,q.)\in\mathfrak B\times\mathfrak Q}w^{\boldsymbol\rho_0(\beta.,q.)\beta.q.}(z,1;\gamma^{\boldsymbol\rho_0(\beta.,q.)\beta.q.})\\
\ge&\sup_{\boldsymbol\rho\in\mathcal P}\inf_{(\beta.,q.)\in\mathfrak B\times\mathfrak Q}w^{\boldsymbol\rho(\beta.,q.)\beta.q.}(z,1;\gamma^{\boldsymbol\rho(\beta.,q.)\beta.q.})-\epsilon\\
=&\sup_{\boldsymbol\rho\in\mathcal P}\inf_{(\beta.,q.)\in\mathfrak B\times\mathfrak Q}w^{\boldsymbol\rho(\beta.,q.)\beta.q.}(z;\gamma^{\boldsymbol\rho(\beta.,q.)\beta.q.})-\epsilon.
\end{align*}
The lemma is proved.

\end{proof}

\section{Proof of Theorem {\ref{thm1}}}\label{section4}

\begin{proof}[Proof of (\ref{thm1})] 

As discussed in Remark \ref{remark1}, by Lemma \ref{lemma1}, to prove that the upper value function $v^+(\cdot)$ is locally Lipschitz in $D$, it suffices to obtain first derivative estimates in all directions on $v^{(\delta)}$ which are independent of $\delta$. From here to the end of the proof, we abbreviate $v^{(\delta)}$, $(f^{\alpha\beta})^{(\delta)}$, $g^{(\delta)}$ to $v$, $f^{\alpha\beta}$, $g$, respectively. Also, keep in mind that all estimates in the proof should be, and actually are, independent of $\delta$, i.e. the uniform lower bound of the smallest eigenvalue of $a^{\alpha\beta}$.

We first fix the small positive $\kappa$ and pick $x\in D_\kappa$ and $\xi\in\Rd$. Choose a sufficiently small positive $ \epsilon_0$, such that $B(x,\epsilon_0|\xi|):=\{y:|y-x|\le \epsilon_0|\xi|\}\subset D_\kappa$. For any $ \epsilon\in(0, \epsilon_0)$, by Lemma \ref{lemma3},
\begin{align*}
v(x+\epsilon\xi)=&\sup_{\boldsymbol\rho\in\mathcal P} \inf_{(\beta.,q.)\in\mathfrak B\times\mathfrak Q}E_{x+\epsilon\xi}^{\boldsymbol\rho(\beta.,q.)\beta.q.}\bigg[v(z_\gamma(\epsilon))p_\gamma(\epsilon)e^{-\phi_\gamma(\epsilon)}\\
&+\int_0^\gamma\theta_t(\epsilon)f^{\alpha_t\beta_t}(z_t(\epsilon))p_t(\epsilon)e^{-\phi_t(\epsilon)}dt\bigg],
\end{align*}
where $z_t^{\alpha.\beta.q.,x+\epsilon\xi}$ is defined by (\ref{itoz}). We claim that for all $\gamma^{\alpha.\beta.,x+\epsilon\xi}\le \bar\tau^{\alpha.\beta.,x+\epsilon\xi}_\kappa(\epsilon)\wedge \hat\tau^{\alpha.\beta.q.,x+\epsilon\xi}_\kappa(\epsilon),$
\begin{align}
\label{bds}v(x+\epsilon\xi)\le&\sfm\ifb E_{x+\epsilon\xi}^{\boldsymbol\mu(\beta.)\beta.}\bigg[v(y_\gamma(\epsilon))p_\gamma(\epsilon)e^{-\phi_\gamma(\epsilon)}\\
&+\int_0^\gamma\theta_t(\epsilon)f^{\alpha_t\beta_t}(y_t(\epsilon))p_t(\epsilon)e^{-\phi_t(\epsilon)}dt\bigg].\nonumber
\end{align}
Indeed, notice that for all $\alpha.\in\mathfrak A$ and $\beta.\in\mathfrak B$, the process $Q_s^{\alpha.\beta.}(\epsilon)$ defined by (\ref{Qs}) is an element of $\mathfrak Q$. Recall that $\mathfrak Q$ is defined in the statement of Lemma \ref{lemma3}. It turns out that 
\begin{align*}
\MoveEqLeft\{y_{t\wedge \gamma^{\alpha.\beta.,x+\epsilon\xi}}^{\alpha.\beta.,x+\epsilon\xi}(\epsilon):\alpha.\in\mathfrak A, \beta.\in\mathfrak B\}\\
&\subset\{z_{t\wedge\gamma^{\alpha.\beta.,x+\epsilon\xi}}^{\alpha.\beta.q.,x+\epsilon\xi}(\epsilon):\alpha.\in\mathfrak A, \beta.\in\mathfrak B,q.\in\mathfrak Q\},
\end{align*}
which implies that for each $\alpha.\in\mathfrak A$,
\begin{align*}
\MoveEqLeft \inf_{(\beta.,q.)\in\mathfrak B\times\mathfrak Q}E_{x+\epsilon\xi}^{\alpha.\beta.q.}\bigg[v(z_\gamma(\epsilon))p_\gamma(\epsilon)e^{-\phi_\gamma(\epsilon)}+\int_0^\gamma\theta_t(\epsilon)f^{\alpha_t\beta_t}(z_t(\epsilon))p_t(\epsilon)e^{-\phi_t(\epsilon)}dt\bigg]\\
&\le\ifb E_{x+\epsilon\xi}^{\alpha.\beta.}\bigg[v(y_\gamma(\epsilon))p_\gamma(\epsilon)e^{-\phi_\gamma(\epsilon)}+\int_0^\gamma\theta_t(\epsilon)f^{\alpha_t\beta_t}(y_t(\epsilon))p_t(\epsilon)e^{-\phi_t(\epsilon)}dt\bigg].
\end{align*}
Then we notice that each strategy in $\mathcal P$ is also a strategy in $\mathcal M$, which implies that from the above inequality, we have
$$\sup_{\boldsymbol\rho\in\mathcal P}(LHS)\le\sup_{\boldsymbol\rho\in\mathcal P}(RHS)\le \sup_{\boldsymbol \mu\in\mathcal M}(RHS).$$
Therefore (\ref{bds}) is true. 

To make the expression shorter, for any $\bar x=(x,x^{d+1},x^{d+2},x^{d+3})\in \bar D\times[0,\infty)\times[0,\infty)\times\mathbb R$, we introduce
\begin{equation}\label{V}
V(\bar x)=v(x)\exp(-x^{d+1})x^{d+2}+x^{d+3}.
\end{equation}
We also define
\begin{align*}
\bar y_t^{\alpha.\beta.,y}(\epsilon)=&\Big(y_t^{\alpha.\beta.,y}(\epsilon), \phi_t^{\alpha.\beta.}(\epsilon),p_t^{\alpha.\beta.}(\epsilon),F_t^{\alpha.\beta.,y}(\epsilon)\Big),\\
\bar x_t^{\alpha.\beta.,x}=&\bar y_t^{\alpha.\beta.,x}(0),
\end{align*}
where $\phi_t^{\alpha.\beta.}(\epsilon)$ and $p_t^{\alpha.\beta.}(\epsilon)$ are given by (\ref{phit}) and (\ref{pt}), and
\begin{equation}\label{Ft}
F_t^{\alpha.\beta.,y}(\epsilon)=\int_0^{t}\theta_s^{\alpha.\beta.}(\epsilon)f^{\alpha_s\beta_s}(y_s^{\alpha.\beta.,y}( \epsilon))p^{\alpha.\beta.}_s( \epsilon)e^{-\phi_s^{\alpha.\beta.}( \epsilon)}ds.
\end{equation}
Then for the stopping times 
$$\gamma^{\alpha.\beta.}:=\bar\tau_\kappa^{\alpha.\beta.,x+\epsilon\xi}(\epsilon)\wedge\tau_\kappa^{\alpha.\beta,x}\wedge T\wedge\vartheta^{\alpha.\beta.,\xi}_n,$$
where $T\in [1,\infty)$ is constant and 
$$\vartheta_n^{\alpha.\beta.,\xi}=\tau_\kappa^{\alpha.\beta.,x}\wedge\inf\{t\ge0:|\xi^{\alpha.\beta.,\xi}_t|\ge n\},$$
from (\ref{bds}) and (\ref{dpp}), we have
\begin{align*}
v(x+ \epsilon\xi)\le&\sfm\ifb E^{\boldsymbol\mu(\beta.)\beta.}_{x+\epsilon\xi}V(\bar y_\gamma(\epsilon)),\\
v(x)=&\sfm\ifb E^{\boldsymbol\mu(\beta.)\beta.}_xV(\bar x_\gamma),
\end{align*}
respectively.

By noticing that
\begin{align*}
\MoveEqLeft\sfm\ifb E^{\boldsymbol\mu(\beta.)\beta.}_{x+\epsilon\xi}V(\bar y_\gamma(\epsilon))-\sfm\ifb E^{\boldsymbol\mu(\beta.)\beta.}_xV(\bar x_\gamma) \\
&\le\sfm\Big(\ifb E^{\boldsymbol\mu(\beta.)\beta.}_{x+\epsilon\xi}V(\bar y_\gamma(\epsilon))-\ifb E^{\boldsymbol\mu(\beta.)\beta.}_xV(\bar x_\gamma)\Big) \\
&\le\sfm\sfb\Big(E^{\boldsymbol\mu(\beta.)\beta.}_xV(\bar x_\gamma)-E^{\boldsymbol\mu(\beta.)\beta.}_{x+\epsilon\xi}V(\bar y_\gamma(\epsilon))\Big) \\
&\le\sfa\sfb\Big(E^{\alpha.\beta.}_xV(\bar x_\gamma)-E^{\alpha.\beta.}_{x+\epsilon\xi}V(\bar y_\gamma(\epsilon))\Big),
\end{align*}
we have
\begin{align}
\frac{v(x+ \epsilon\xi)-v(x)}{ \epsilon}\le&\sfa\sfb E\bigg|\frac{V(\bar y^{\alpha.\beta.,x+\epsilon\xi}_{\gamma^{\alpha.\beta.}}(\epsilon))-V(\bar x^{\alpha.\beta.,x}_{\gamma^{\alpha.\beta.}})}{ \epsilon}\bigg|\nonumber\\
\le &I_1(\epsilon,T,n)+I_2(\epsilon,T,n).\label{i1i2}
\end{align}
Here
\begin{align*}
I_1(\epsilon, T,n)=&\sfa\sfb E\bigg|\frac{V(\bar y^{\alpha.\beta.,x+\epsilon\xi}_{\gamma^{\alpha.\beta.}}(\epsilon))-V(\bar x^{\alpha.\beta.,x}_{\gamma^{\alpha.\beta.}})}{ \epsilon}-V_{(\bar\xi_{\gamma^{\alpha.\beta.}}^{\alpha.\beta.,\xi})}(\bar x_{\gamma^{\alpha.\beta.}}^{\alpha.\beta.,x})\bigg|,\\
I_2(\epsilon, T,n)=&\sfa\sfb E|V_{(\bar\xi_{\gamma^{\alpha.\beta.}}^{\alpha.\beta.,\xi})}(\bar x_{\gamma^{\alpha.\beta.}}^{\alpha.\beta.,x})|,
\end{align*}
where
\begin{equation}\label{xibar}
\bar\xi_t^{\alpha.\beta.,\xi}=(\xi_t^{\alpha.\beta.,\xi},\xi_t^{d+1,\alpha.\beta.},\xi_t^{d+2,\alpha.\beta.},\xi_t^{d+3,\alpha.\beta.}),
\end{equation}
with $\xi_t^{\alpha.\beta.,\xi}$ the solution to the stochastic differential equation (\ref{itoxi}) and
\begin{align}
\xi_t^{d+1,\alpha.\beta.}=&\int_0^t2r^{\alpha.\beta.}_sc^{\alpha_s\beta_s}ds,\\
\xi_t^{d+2,\alpha.\beta.}=&\int_0^t\pi^{\alpha.\beta.}_sdw_s\Big(=\zeta^{\alpha.\beta.,0}_t\Big),\\
\xi_t^{d+3,\alpha.\beta.}=&\int_0^te^{-\phi_s^{\alpha.\beta.}}\Big[f^{\alpha_s\beta_s}_{(\xi_s^{\alpha.\beta.,\xi})}(x_s^{\alpha.\beta.,x})\label{42}\\
&+\big(2r^{\alpha.\beta.}_s-\xi_s^{d+1,\alpha.\beta.}+\xi_s^{d+2,\alpha.\beta.}\big)f^{\alpha_s\beta_s}(x_s^{\alpha.\beta.,x})\Big]ds.\nonumber
\end{align}

We claim that 
\begin{equation}\label{i1}
\lim_{\epsilon\downarrow0}I_1(\epsilon,T,n)=0.
\end{equation}
Since $V(\bar x)$ is globally $C^{1,\gamma}$-continuous, to apply (c4) of Lemma \ref{lemma2}, it suffices to prove that
\begin{equation}
\lim_{\epsilon\downarrow0}\bigg(\sfa\sfb E\sup_{t\le\gamma}\Big|\frac{\bar y_t^{\alpha.\beta.,x+\epsilon\xi}(\epsilon)-\bar x_t^{\alpha.\beta.,x}}{\epsilon}-\bar\xi_t^{\alpha.\beta.,\xi}\Big|\bigg)=0. \label{barlimy}
\end{equation}
In other words, we just need to show
\begin{align}
&\lim_{\epsilon\downarrow0}\bigg(\sfa\sfb E\sup_{t\le\gamma}\Big|\frac{ y_t^{\alpha.\beta.,x+\epsilon\xi}(\epsilon)- x_t^{\alpha.\beta.,x}}{\epsilon}-\xi_t^{\alpha.\beta.,\xi}\Big|\bigg)=0,\label{lim1}\\
&\lim_{\epsilon\downarrow0}\bigg(\sfa\sfb E\sup_{t\le\gamma}\Big|\frac{ \phi_t^{\alpha.\beta.}(\epsilon)- \phi_t^{\alpha.\beta.}}{\epsilon}-\xi_t^{d+1,\alpha.\beta.}\Big|\bigg)=0,\label{lim2}\\
&\lim_{\epsilon\downarrow0}\bigg(\sfa\sfb E\sup_{t\le\gamma}\Big|\frac{ p_t^{\alpha.\beta.}(\epsilon)- 1}{\epsilon}-\xi_t^{d+2,\alpha.\beta.}\Big|\bigg)=0,\label{lim3}\\
&\lim_{\epsilon\downarrow0}\bigg(\sfa\sfb E\sup_{t\le\gamma}\Big|\frac{ F_t^{\alpha.\beta.,x+\epsilon\xi}(\epsilon)- F_t^{\alpha.\beta.,x}}{\epsilon}-\xi_t^{d+3,\alpha.\beta.}\Big|\bigg)=0\label{lim4}.
\end{align}

Equation (\ref{lim1}) is true via  (c2) in Lemma \ref{lemma2}. 

Equation (\ref{lim2}) is true because
$$\bigg|\frac{ \phi_t^{\alpha.\beta.,x+\epsilon\xi}(\epsilon)- \phi_t^{\alpha.\beta.,x}}{\epsilon}-\xi_t^{d+1,\alpha.\beta.}\bigg|\le\epsilon T|c^{\alpha\beta}|_{0,A\times B}4\pi (r_t^{\alpha.\beta.})^2$$
via Taylor's theorem.

To prove Equation (\ref{lim3}), we notice that
$$\frac{p_t^{\alpha.\beta.}(\epsilon)-1}{\epsilon}-\xi_t^{d+1,\alpha.\beta.}=\int_0^t\big(p^{\alpha.\beta.}_s(\epsilon)-1\big)\pi^{\alpha.\beta.}_sdw_s.$$
Recall that the stopping time $\gamma^{\alpha.\beta.}$ is bounded by $T\wedge\vartheta_n^{\alpha.\beta.,\xi}$. It follows by Davis inequality that
\begin{align*}
\MoveEqLeft E^{\alpha.\beta.}\sup_{t\le\gamma\wedge\gamma_m}\Big|\frac{ p_t(\epsilon)- 1}{\epsilon}-\xi_t^{d+2}\Big|\\
&=E^{\alpha.\beta.}\sup_{t\le\gamma\wedge\gamma_m}\Big|\int_0^t\big(p_s(\epsilon)-1\big)\pi_sdw_s\Big|\\
&\le3E^{\alpha.\beta.}\Big(\int_0^{\gamma\wedge\gamma_m}\big(p_t(\epsilon)-1\big)^2|\pi_t|^2dt\Big)^{1/2}\\
&\le3\epsilon E^{\alpha.\beta.}\bigg[\sup_{t\le\gamma\wedge\gamma_m}\Big|\frac{p_t(\epsilon)-1}{\epsilon}\Big|\Big(\int_0^{\gamma\wedge\gamma_m}|\pi_t|^2dt\Big)^{1/2}\bigg]\\
&\le3\epsilon\sqrt Tn N\kappa^{-1} E^{\alpha.\beta.}\sup_{t\le\gamma\wedge\gamma_m}\Big|\frac{p_t(\epsilon)-1}{\epsilon}\Big|\\
&\le3\epsilon\sqrt T nN\kappa^{-1} E^{\alpha.\beta.}\sup_{t\le\gamma\wedge\gamma_m}\bigg(\Big|\frac{p_t(\epsilon)-1}{\epsilon}-\xi_t^{d+2}\Big|+|\xi_t^{d+2}|\bigg).
\end{align*}
where $\gamma_m$ is a localizing sequence of stopping times such that the left hand side of the inequalities is finite for each $m$. Collecting similar terms to the left side of the inequality and then letting $m\rightarrow\infty$, by the monotone convergence theorem, we obtain
\begin{align*}
\MoveEqLeft\big(1-3\epsilon\sqrt TnN\kappa^{-1}\big)E^{\alpha.\beta.}\sup_{t\le\gamma}\Big|\frac{ p_t(\epsilon)- 1}{\epsilon}-\xi_t^{d+2}\Big|\\
&\le3\epsilon\sqrt T nN\kappa^{-1} E^{\alpha.\beta.}\Big(\int_0^\gamma|\pi_t|^2dt\Big)^{1/2}.
\end{align*}
Then Equation (\ref{lim3}) is obtained by first taking the supremum over $\mathfrak A\times \mathfrak B$ and then letting $\epsilon\downarrow0$. 

To prove Equation (\ref{lim4}), for each $(\alpha,\beta)\in A\times B$, we introduce the  function:
\begin{equation}\label{F}
G^{\alpha\beta}: \bar D\times [0,\infty)\times[0,\infty)\times\mathbb R\rightarrow \mathbb R;\  \bar x\mapsto f^{\alpha\beta}(x)\exp(-x^{d+1})x^{d+2}. 
\end{equation}
From (\ref{Ft}) and (\ref{42}) we have
\begin{align*}
\MoveEqLeft\frac{ F_t^{\alpha.\beta.,x+\epsilon\xi}(\epsilon)- F_t^{\alpha.\beta.,x}}{\epsilon}-\xi_t^{d+3,\alpha.\beta.}\\
&=\int_0^t\bigg[\frac{ G^{\alpha_s\beta_s}(\bar y_s^{\alpha.\beta.,x+\epsilon\xi}(\epsilon))-G^{\alpha_s\beta_s}( \bar x_s^{\alpha.\beta.,x})}{\epsilon}-G^{\alpha_s\beta_s}_{(\bar\xi_s^{\alpha.\beta.,\xi})}(\bar x_s^{\alpha.\beta.,x})\\
&\quad+\frac{\arctan(\pi2\epsilon r_s^{\alpha.\beta.})}{\epsilon\pi}G^{\alpha_s\beta_s}(\bar y_s^{\alpha.\beta.,x+\epsilon\xi}(\epsilon))-2r_s^{\alpha.\beta.} G^{\alpha.\beta.}(\bar x^{\alpha.\beta.,x}_s)\bigg]ds\\
:&=\int_0^t\Big[(H_1)_s^{\alpha.\beta.}+(H_2)_s^{\alpha.\beta.}\Big]ds
\end{align*}
To prove (\ref{lim4}) it suffices to show that
\begin{equation*}\label{lim5}
\lim_{\epsilon\downarrow0}\bigg(\sup_{\alpha\in\mathfrak A}\sfb E\sup_{t\le\gamma}\big|(H_1)_t^{\alpha.\beta.}\big|\bigg)=\lim_{\epsilon\downarrow0}\bigg(\sup_{\alpha\in\mathfrak A}\sfb E\sup_{t\le\gamma}\big|(H_2)_t^{\alpha.\beta.}\big|\bigg)=0,
\end{equation*}
\begin{equation*}\label{lim6}
\end{equation*}
which are valid due to (\ref{lim1}) - (\ref{lim3}), (c3) and (c4) in Lemma \ref{lemma2} and
$$
\Big|\frac{\theta_t^{\alpha.\beta.}(\epsilon)-1}{\epsilon}-2r_t^{\alpha.\beta.}\Big|\le 4\epsilon\pi (r_t^{\alpha.\beta.})^2.
$$
Therefore (\ref{lim4}) is proved, and we obtain (\ref{i1}). 

Next, we estimate $I_2(\epsilon, T,n)$. 

From (\ref{V}) we have
$$V_{(\bar\xi)}(\bar x)=v_{(\xi)}(x)e^{-x^{d+1}}x^{d+2}+v(x)e^{-x^{d+1}}(\xi^{d+2}-x^{d+2}\xi^{d+1})+\xi^{d+3}.$$
As a result,
$$V_{(\bar\xi^{\alpha.\beta.,\xi}_t)}(\bar x^{\alpha.\beta.,x}_t)=e^{-\phi_t^{\alpha.\beta.,x}}v_{(\xi_t^{\alpha.\beta.,\xi})}(x_t^{\alpha.\beta.,x})+X_t^{\alpha.\beta.,x,\xi},$$
where
\begin{align*}
X_t^{\alpha.\beta.,x,\xi}=&e^{-\phi_t^{\alpha.\beta.,x}}(\xi^{d+2,\alpha.\beta.}_t-\xi_t^{d+1,\alpha.\beta.})v(x_t^{\alpha.\beta.,x})+\xi_t^{d+3,\alpha.\beta.}.
\end{align*}
It follows that
\begin{align*}
I_2(\epsilon, T,n)&=\sup_{\alpha.\in\mathfrak A}\sfb E^{\alpha.\beta.}_{x,\xi}|V_{(\bar\xi_\gamma)}(\bar x_\gamma)|\\
& \le \sup_{\alpha.\in\mathfrak A}\sfb E^{\alpha.\beta.}_{x,\xi} |v_{(\xi_{\gamma})}(x_{\gamma})|+\sup_{\alpha.\in\mathfrak A}\sfb E^{\alpha.\beta.}_{x,\xi}|X_{\gamma}|.
\end{align*}

We first claim that
\begin{equation}\label{claimx}
\sfa\sfb E^{\alpha.\beta.}_{x,\xi}|X_{\gamma}|\le N\overline{\mathrm B}^{1/2}(x,\xi),
\end{equation}
where $N$ is independent of $\epsilon$, $T$, $n$.
Indeed, from the definition of $X_t$,
\begin{align*}
|X_\gamma|\le&|v|_{0,D}\big(|\xi_\gamma^{d+1}|+|\xi_\gamma^{d+2}|\big)\\
&+|f|_{1,D}\bigg[\int_0^\gamma e^{-\phi_t}\Big(|\xi_t|+|\xi_t^{d+1}|+|\xi_t^{d+2}|+2|r_t|\Big)dt\bigg].
\end{align*}
Notice that we have the following estimates:
\begin{align*}\allowdisplaybreaks
\MoveEqLeft|v|_{0,D}\le |g|_{0,D}+|\psi|_{0,D}\sup_{(\alpha,\beta)\in A\times B}|f^{\alpha\beta}|_{0,D},\\
\MoveEqLeft\sfa\sfb E^{\alpha.\beta.}|\xi_\gamma^{d+1}|\\
&\le K\sfa\sfb E^{\alpha.\beta.}\int_0^\gamma\Big(|\xi_t|+\frac{|\psi_{(\xi_t)}|}{\psi}\Big)dt\\
&\le N\overline{\mathrm B}^{1/2}(x,\xi),\\
\MoveEqLeft\sfa\sfb E^{\alpha.\beta.}|\xi_\gamma^{d+2}|\\
&\le3\sfa\sfb E^{\alpha.\beta.}\langle\xi^{d+2}\rangle_\gamma^{1/2}\\
&\le K\sfa \sfb E^{\alpha.\beta.}\bigg(\int_0^\gamma\Big(|\xi_t|^2+\pxtsops \Big)dt\bigg)^{1/2}\\
&\le N\overline{\mathrm B}^{1/2}(x,\xi),\\
\MoveEqLeft\sfa\sfb E^{\alpha.\beta.}\int_0^\gamma\big(|\xi_t|+|r_t|\big)dt\\
&\le K\sfa\sfb E^{\alpha.\beta.}\int_0^\gamma\Big(|\xi_t|+\frac{|\psi_{(\xi_t)}|}{\psi}\Big)dt\\
&\le N\overline{\mathrm B}^{1/2}(x,\xi),\\
\MoveEqLeft\sfa\sfb E^{\alpha.\beta.}\int_0^\gamma e^{-\phi_t}|\xi_t^{d+1}|dt\\
&\le\sfa\sfb E^{\alpha.\beta.}\int_0^\gamma2|r_t|dt\int_0^\gamma e^{-ct}cdt\\
&\le K\sfa\sfb E^{\alpha.\beta.}\int_0^\gamma\Big(|\xi_t|+\frac{|\psi_{(\xi_t)}|}{\psi}\Big)dt\\
&\le N\overline{\mathrm B}^{1/2}(x,\xi),\\
\MoveEqLeft\sfa\sfb E^{\alpha.\beta.}\int_0^\gamma |\xi_t^{d+2}|dt\\
&\le\sfa\sfb E^{\alpha.\beta.}\gamma\sup_{t\le\gamma}|\xi_t^{d+2}|\\
&\le\sfa \sfb\big(E^{\alpha.\beta.} \gamma^2\big)^{1/2} \big(E^{\alpha.\beta.} \sup_{t\le\gamma}|\xi_t^{d+2}|^2\big)^{1/2}\\
&\le4|\psi|_{0,D}\sfa\sfb \big(E^{\alpha.\beta.} \langle\xi^{d+2}\rangle_\gamma\big)^{1/2}\\
&\le N\overline{\mathrm B}^{1/2}(x,\xi).
\end{align*}
Applying the estimates above, (\ref{claimx}) is proved.

We also claim that
\begin{align}\label{250}
\MoveEqLeft\varlimsup_{n\uparrow\infty}\varlimsup_{T\uparrow\infty}\varlimsup_{\epsilon\downarrow0}\sfa\sfb E^{\alpha.\beta}_{x,\xi} |v_{(\xi_{\gamma})}(x_{\gamma})|\\
&\le\bigg(\sup_{\substack{y\in\{\psi=\kappa\}\\\eta\in\Rd\setminus\{0\}}}\frac{|v_{(\eta)}(y)|}{\sqrt{\mathrm B_1(y,\eta)}}+2\bigg)\sqrt{2\overline{\mathrm B}(x,\xi)}.\nonumber
\end{align}
Indeed, we notice that
\begin{align*}
\sfa\sfb E^{\alpha.\beta.}_{x,\xi}|v_{(\xi_{\gamma})}(x_{\gamma})|=&\sfa\sfb E^{\alpha.\beta.}_{x,\xi}\frac{|v_{(\xi_{\gamma})}(x_{\gamma})|}{\sqrt{\underline{\mathrm{B}}(x_{\gamma},\xi_{\gamma})}}\cdot\sqrt{\underline{\mathrm{B}}(x_{\gamma},\xi_{\gamma})}\\
\le&J_1(\epsilon, T,n)+J_2(\epsilon, T,n),
\end{align*}
where
\begin{align*}
J_1(\epsilon, T,n)=&\sfa\sfb E^{\alpha.\beta.}_{x,\xi}\bigg(\frac{|v_{(\xi_{\gamma})}(x_{\gamma})|}{\sqrt{\underline{\mathrm{B}}(x_{\gamma},\xi_{\gamma})}}-\frac{|v_{(\xi_{\gamma})}(x_{\tau_\kappa})|}{\sqrt{\underline{\mathrm{B}}(x_{\tau_\kappa},\xi_{\gamma})}}\bigg)\sqrt{\underline{\mathrm{B}}(x_{\gamma},\xi_{\gamma})},\\
J_2(\epsilon, T,n)=&\sfa\sfb E^{\alpha.\beta.}_{x,\xi}\frac{|v_{(\xi_{\gamma})}(x_{\tau_\kappa})|}{\sqrt{\underline{\mathrm{B}}(x_{\tau_\kappa},\xi_{\gamma})}}\sqrt{\underline{\mathrm{B}}(x_{\gamma},\xi_{\gamma})}.
\end{align*}
Note that
$$v_{(\xi)}(x)/\sqrt{\underline{\mathrm{B}}(x,\xi)}=v_{(\xi/|\xi|)}(x)/\sqrt{\underline{\mathrm{B}}(x,\xi/|\xi|)}$$
is a continuous function from $\bar D_\kappa\times S_1$ to $\mathbb{R}$, where $S_1$ is the unit sphere in $\Rd$. By Weierstrass approximation theorem,  there exists a polynomial $W(x,\xi): \bar D_\kappa\times S_1\rightarrow\mathbb{R}$, such that 
$$\sup_{x\in D_\kappa,\xi\in S_1}\Big|\frac{v_{(\xi)}(x)}{\sqrt{\underline{\mathrm{B}}(x,\xi)}}-W(x,\xi)\Big|\le1.$$
It follows that
\begin{align*}
J_1(\epsilon,T,n)\le&\sfa\sfb E^{\alpha.\beta.}_{x,\xi}\big|W(x_{\gamma},\xi_{\gamma}/|\xi_\gamma|)-W(x_{\tau_\kappa},\xi_{\gamma}/|\xi_\gamma|)\big|\sqrt{\underline{\mathrm{B}}(x_{\gamma},\xi_{\gamma})}\\
&+2\sfa\sfb E^{\alpha.\beta.}_{x,\xi}\sqrt{\underline{\mathrm{B}}(x_{\gamma},\xi_{\gamma})}\\
\le&(N\kappa^{-1})\sfa\sfb E^{\alpha.\beta.}_{x,\xi}|x_{\gamma}-x_{\tau_\kappa}||\xi_\gamma|(\mathbbm1_{\tau_\kappa\le\vartheta_n}+\mathbbm1_{\tau_\kappa>\vartheta_n})\\&+2\sqrt{2\overline{\B}(x,\xi)}\\
\le&(Nn\kappa^{-1})\sfa\sfb E^{\alpha.\beta.}_{x}\Big[(\tau_\kappa-\gamma)+\sqrt{\tau_\kappa-\gamma}\Big]\\
&+(N\kappa^{-1})\sfa\sfb E^{\alpha.\beta.}_{x,\xi}|\xi_\gamma|\mathbbm1_{\tau_\kappa>\vartheta_n}+2\sqrt{2\overline{\B}(x,\xi)}.
\end{align*}
Notice that
\begin{align*}
\MoveEqLeft E_x^{\alpha.\beta.}(\tau_\kappa-\gamma)\\
&\le E(\tau_\kappa^{\alpha.\beta.,x}-\tau_\kappa^{\alpha.\beta.,x}\wedge\tau_\kappa^{\alpha.\beta.,x+\epsilon\xi})+E(\tau_\kappa^{\alpha.\beta.,x}-\tau_\kappa^{\alpha.\beta.,x}\wedge T),\\
\MoveEqLeft E^{\alpha.\beta.}_{\xi}|\xi_\gamma|\mathbbm1_{\tau_\kappa>\vartheta_n}\\
&\le\sqrt{E^{\alpha.\beta.}_{x,\xi}|\xi_\gamma|^2}\sqrt{P^{\alpha.\beta.}_{x,\xi}\Big(\sup_{t\le\tau_\kappa}|\xi_t|\ge n\Big)}\\
&\le \frac{1}{n}E^{\alpha.\beta.}_\xi\sup_{t\le\tau_\kappa}|\xi_t|^2.
\end{align*}
Thus by (d1), (d2) and (a1) in Lemma \ref{lemma2},
$$\varlimsup_{n\uparrow\infty}\varlimsup_{T\uparrow\infty}\varlimsup_{\epsilon\downarrow0}J_1(\epsilon,T,n)\le2\sqrt{2\overline{\B}(x,\xi)}.$$
Also, notice that
\begin{align*}
J_2(\epsilon,T,n)\le&\sup_{\substack{y\in\{\psi=\kappa\}\\\eta\in  \Rd\setminus\{0\}}}\frac{|v_{(\eta)}(y)|}{\sqrt{\mathrm{B}_1(y,\eta)}}\cdot\sqrt{2\overline{\B}(x,\xi)}.
\end{align*}
Thus (\ref{250}) is proved.

Combining (\ref{claimx}) and (\ref{250}), we obtain
\begin{equation}\label{i2}
\varlimsup_{n\uparrow\infty}\varlimsup_{T\uparrow\infty}\varlimsup_{\epsilon\downarrow0}I_2(\epsilon,T,n)\le\bigg(\sup_{\substack{y\in \{\psi=\kappa\}\\\eta\in  \Rd\setminus\{0\}}}\frac{|v_{(\eta)}(y)|}{\sqrt{\mathrm{B}_1(y,\eta)}}+N\bigg)\sqrt{\overline{\B}(x,\xi)}.
\end{equation}

It remains to let $\kappa\downarrow0$ and compute
$$
\varlimsup_{\kappa\downarrow0}\bigg(\sup_{\substack{x\in\{\psi=\kappa\}\\\xi\in\Rd\setminus\{0\}}}\vxosbw\bigg).
$$
Due to the compactness of $(\partial D_\kappa)\times S_1$, for each $\kappa$, there exist $x(\kappa)\in\partial D_\kappa$ and $\xi(\kappa)\in S_1$, such that
$$
\sup_{\substack{x\in\{\psi=\kappa\}\\\xi\in\Rd\setminus\{0\}}}\vxosbw=\frac{|v_{(\xi(\kappa))}(x(\kappa))|}{\sqrt{\mathrm{B}_1(x(\kappa),\xi(\kappa))}}.
$$
A subsequence of $(x(\kappa),\xi(\kappa))$ converges to some $(y,\eta)$, where $y\in\partial D$ and $|\eta|=1$.

If $\psi_{(\eta)}(y)\ne0$, then $\mathrm{B}_1(x(\kappa),\xi(\kappa))\nearrow\infty$ as $\kappa\searrow0$. In this case,
$$
\varlimsup_{\kappa\downarrow0}\bigg(\sup_{\substack{x\in\{\psi=\kappa\}\\\xi\in\Rd\setminus\{0\}}}\vxosbw\bigg)=\varlimsup_{\kappa\downarrow0}\frac{|v_{(\xi(\kappa))}(x(\kappa))|}{\sqrt{\mathrm{B}_1(x(\kappa),\xi(\kappa))}}=0.
$$

If $\psi_{(\eta)}(y)=0$, then $\eta$ is tangent to $\partial D$ at $y$. In this case,
\begin{equation*}\label{lims}
\varlimsup_{\kappa\downarrow0}\bigg(\sup_{\substack{x\in\{\psi=\kappa\}\\\xi\in\Rd\setminus\{0\}}}\vxosbw\bigg)=\varlimsup_{\kappa\downarrow0}\frac{|v_{(\xi(\kappa))}(x(\kappa))|}{\sqrt{\mathrm{B}_1(x(\kappa),\xi(\kappa))}}=\frac{|g_{(\eta)}(y)|}{\lambda}.
\end{equation*}

Therefore for all $x\in D$ and $\xi\in\Rd$, we have
\begin{align*}
v_{(\xi)}(x)&\le N\sqrt{\mathbbm1_{x\in D_\kappa^\lambda}B_1(x,\xi)+\mathbbm1_{x\in D_{\lambda^2}}B_2(x,\xi)}\\
&\le N(|\xi|+|\psi_{(\xi)}|\psi^{-1/2}).
\end{align*}

Replacing $\xi$ with $-\xi$ in the inequality above, we obtain the first derivative estimate. 

\end{proof}

\section{Applications to the Dirichlet problem for nonconvex degenerate elliptic Hessian equations}\label{section5} As an example of application, we study the Dirichlet problem for certain nonconvex degenerate elliptic Hessian equations.

Let $d$ be the dimension of the spatial space and $D$ be a smooth bounded, strictly convex domain. For $1\le i\le d$, let $\lambda_k(\gamma)$ be the $i$-th smallest eigenvalue of the matrix $\gamma\in \mathbb S^d$.

\subsection{Example 1}
Consider the Dirichlet problem 
\begin{equation}\label{example1}
\left\{
\begin{array}{rcll}
H({u_{xx}})&=&f &\text{in } D\\
u&=&g &\text{on }\partial D,
\end{array}
\right. 
\end{equation}
with
\begin{equation}\label{2sum}
H(u_{xx})=\sum_{i=1}^{k_1}\lambda_i(u_{xx})+\sum_{i=d-k_2+1}^{d}\lambda_i(u_{xx}),\quad1\le k_1, k_2\le d.
\end{equation}

If $k_1+k_2=d$, then $H$ is the Laplacian.

If $k_1+k_2>d$, then $H$ is nonconvex, uniformly non-degenerate  elliptic, so the well-known $C^{1,\alpha}$-regularity result can be applied to this case.

If $k_1+k_2<d$, then $H$ is nonconvex degenerate  elliptic. We apply Theorem \ref{thm1} to study this case.

Notice that
\begin{align*}
\sum_{i=1}^{k_1}\lambda_i(u_{xx})=&\inf_{\alpha\in P_{k_1}}\big[\tr(\alpha u_{xx})\big],\\
\sum_{i=d-k_2+1}^{d}\lambda_i(u_{xx})=&\sup_{\alpha\in P_{k_2}}\big[\tr(\alpha u_{xx})\big],
\end{align*}
where
\begin{align}\label{proj}
P_k
=&\{\alpha\in \mathbb M^{d\times d}:  \alpha^2=\alpha,\alpha^*=\alpha, \operatorname{rank}(\alpha)=k\}.
\end{align}

Therefore the Hessian equation in the Dirichlet problem (\ref{example1}) can be rewritten as the Isaacs equation
$$\inf_{\beta\in P_{k_1}}\sup_{\alpha\in P_{k_2}}\Big\{\tr\big[(\alpha+\beta)u_{xx}\big]-f\Big\}=0.$$

Since the domain $D$ is bounded smooth and strictly convex, there exists a $C^{\infty}$ global barrier $\psi$ satisfying
\begin{itemize}
\item $D=\{x\in\Rd: \psi>0\}$;
\item $|\psi_x|\ge1$ on $\partial D$;
\item $\tr(a\psi_{xx})\le-1$ in $\bar D, \forall a\in\{\overline{\mathbb S_+^d}:\tr(a)=1\}$.
\end{itemize}
Therefore, we have
$$\sup_{\beta\in P_{k_1}}\sup_{\alpha\in P_{k_2}}\Big\{\tr\big[(\alpha+\beta) \psi_{xx}\big]\Big\}\le-1\mbox{ in }\bar D,\quad\forall 1\le k\le d,$$
which means that Assumption \ref{geod} holds. 

We also observe that for any $q\in \mathbb O^d$, we have $qP_kq^*=P_k, \forall 1\le k\le d$, which implies that Assumption \ref{invc} holds.

Therefore, Theorem 2.1 are applicable to the Dirichlet problem ({\ref{example1}}). Under the settings in Section \ref{section2}, here we particularly let,
\begin{align}
A=P_{k_2},\quad B=P_{k_1},\label{1coe2}
\end{align}
and $\forall \alpha\in A,\beta\in B, x\in D$,
\begin{equation}
a^{\alpha\beta}=\alpha+\beta,\quad \sigma^{\alpha\beta}=\sqrt{2a^{\alpha\beta}},\quad b^{\alpha\beta}=c^{\alpha\beta}=0,\quad f^{\alpha\beta}(x)=f(x).
\label{1coe4}
\end{equation}
We also define $\mathfrak A$, $\mathfrak B$, $\mathcal M$ and $\mathcal N$ accordingly.

\begin{theorem}\label{thm51} Consider the upper value function of the stochastic differential game described by (\ref{1coe2}) and (\ref{1coe4}), i.e.,
\begin{equation*}
v(x)=\sup_{\boldsymbol\mu\in\mathcal M_{}}\inf_{\beta.\in\mathfrak B_{}}E^{\boldsymbol\mu(\beta.)\beta.}_x\bigg[\int_0^{\tau}f(x_t)dt+g(x_\tau)\bigg],
\end{equation*}
with
$$x_t^{\alpha.\beta.,x}=x+\int_0^t\sqrt{2(\alpha_s+\beta_s)}dw_s.$$
For any $f\in C^{0,1}(\bar D)$ and $g\in C^{0,1}(\partial D)$, the value function $v$ is in the class of $C^{0,1}_{loc}(D)\cap C(\bar D)$, and it is the unique viscosity solution to the Dirichlet problem (\ref{example1}). Moreover, for a.e. $x\in D$,
\begin{equation*}
\big|v_{(\xi)}\big|\le N\big(|\xi|+|\psi_{(\xi)}|\psi^{-1/2}\big),\quad\forall\xi\in\Rd,
\end{equation*}
where $\psi$ is a global barrier of the domain $D$, and the constant $N$ depends only on $|f|_{0,1,D}$,$|g|_{0,1,\partial D}$,$|\psi|_{3,D}$ and $d$.
\end{theorem}

\begin{remark}
If in (\ref{2sum}), one of the summation disappears, and $k_1,k_2<d$, then $H$ is convex/concave degenerate elliptic. Applying Theorem 2.3 in \cite{InteriorRegularityI}, we know that for $f\in C^{1,1}(\bar D)$, $g\in C^{1,1}(\partial D)$, the Dirichlet problem (\ref{example1}) is unique solvable in the class of $C^{1,1}_{loc}(D)\cap C(\bar D)$.
\end{remark}

\subsection{Example 2} 
Consider the Dirichlet problem 
\begin{equation}\label{example2}
\left\{
\begin{array}{rcll}
H({u_{xx}})&=&f &\text{in } D\\
u&=&g &\text{on }\partial D,
\end{array}
\right. 
\end{equation}
with
\begin{equation}\label{2sum}
H(u_{xx})=\sum_{i=k+1}^{k+j}\lambda_i(u_{xx}),\quad0< k<k+j < d.
\end{equation}

Here, $H$ is still nonconvex, degenerate elliptic.

Observe that
$$\sum_{i=k+1}^{k+j}\lambda_i(u_{xx})=\inf_{\beta\in P_{k+j}}\sup_{\alpha\in P_j}\big[\tr(\beta\alpha\beta^*u_{xx})\big],$$
where $P_{k+j}$ and $P_{j}$ are defined by (\ref{proj}).  

Notice that if $k+2j>d$, then $\operatorname{rank}(\beta\alpha)\ge1$. Consequently, Assumption \ref{geod} are satisfied with the same global barrier in the previous example. Also, for all $q\in \mathbb O^d$, $qP_kq^*=P_k$, so Assumption \ref{invc} holds with the satisfaction of Condition (\ref{othinv}).

Therefore, Theorem 2.1 are applicable to the Dirichlet problem ({\ref{example2}}) when $k+2j>d$. Under the settings in Section \ref{section2}, here we let, in particular,
\begin{align}
A=P_{j},\quad B=P_{k+j},\label{1coe6}
\end{align}
and $\forall \alpha\in A,\beta\in B, x\in D$,
\begin{equation}
a^{\alpha\beta}=\beta\alpha\beta^*,\quad \sigma^{\alpha\beta}=\sqrt{2a^{\alpha\beta}},\quad b^{\alpha\beta}=c^{\alpha\beta}=0,\quad f^{\alpha\beta}(x)=f(x).
\label{1coe7}
\end{equation}
We also define $\mathfrak A$, $\mathfrak B$, $\mathcal M$ and $\mathcal N$ accordingly.

\begin{theorem}\label{thm52} Assume $k+2j>d$. Consider the upper value function of the stochastic differential game described by (\ref{1coe6}) and (\ref{1coe7}), i.e.,
\begin{equation*}
v(x)=\sup_{\boldsymbol\mu\in\mathcal M_{}}\inf_{\beta.\in\mathfrak B_{}}E^{\boldsymbol\mu(\beta.)\beta.}_x\bigg[\int_0^{\tau}f(x_t)dt+g(x_\tau)\bigg],
\end{equation*}
with
$$x_t^{\alpha.\beta.,x}=x+\int_0^t\sqrt{2\beta\alpha\beta^*}dw_s.$$
For any $f\in C^{0,1}(\bar D)$ and $g\in C^{0,1}(\partial D)$, the value function $v$ is in the class of $C^{0,1}_{loc}(D)\cap C(\bar D)$, and it is the unique viscosity solution to the Dirichlet problem (\ref{example2}). Moreover, for a.e. $x\in D$,
\begin{equation*}
\big|v_{(\xi)}\big|\le N\big(|\xi|+|\psi_{(\xi)}|\psi^{-1/2}\big),\quad\forall\xi\in\Rd,
\end{equation*}
where $\psi$ is a global barrier of the domain $D$, and the constant $N$ depends on on $|f|_{0,1,D}$,$|g|_{0,1,\partial D}$,$|\psi|_{3,D}$ and $d$.
\end{theorem}

\providecommand{\bysame}{\leavevmode\hbox to3em{\hrulefill}\thinspace}
\providecommand{\MR}{\relax\ifhmode\unskip\space\fi MR }
\providecommand{\MRhref}[2]{%
  \href{http://www.ams.org/mathscinet-getitem?mr=#1}{#2}
}
\providecommand{\href}[2]{#2}

\end{document}